\documentclass[a4paper,11pt]{article}
\usepackage{amsfonts}
\usepackage{amsthm}
\usepackage{amsmath}
\usepackage{latexsym}
\usepackage{amssymb}
\usepackage[ansinew]{inputenc}
\usepackage{graphicx}
\usepackage{dsfont}
\newcommand{\nui}[1]{N(#1)}

\newtheorem{fed}{Definition}[section]
\newtheorem*{fed*}{Definition}
\newtheorem*{feds*}{Definitions}
\newtheorem{teo}[fed]{Theorem}
\newtheorem*{teo*}{Theorem}
\newtheorem{lem}[fed]{Lemma}
\newtheorem{cor}[fed]{Corollary}
\newtheorem{pro}[fed]{Proposition}
\theoremstyle{definition}
\newtheorem{rem}[fed]{Remark}
\newtheorem{conj}[fed]{Conjecture}
\newtheorem*{rems*}{Remarks}

\newtheorem{nota}[fed]{Notation}

\def\coma{\, , \, }

\def\py{\peso{and}}
\newcommand{\peso}[1]{ \quad \text{ #1 } \quad }

\def\n0{n_{ \text{\rm \tiny o}}}

\newcommand{\IN}[1]{\mathbb {I} _{#1}}

\def\suml{\sum\limits}

\def\bce{\begin{center}}
\def\ece{\end{center}}

\def\cO{{\mathcal O}}
\def\cD{\mathcal D}

\def\py{\peso{and}}
\def\rk{\text{\rm rk}}
\def\noi{\noindent}
\def\cF{\mathcal F}
\def\cG{\mathcal G}

\def\EOE{\hfill $\triangle$}

\def\uno{\mathds{1}}

\def\bm{\left[\begin{array}}
\def\em{\end{array}\right]}
\def\ben{\begin{enumerate}}
\def\een{\end{enumerate}}
\def\bit{\begin{itemize}}
\def\eit{\end{itemize}}
\def\barr{\begin{array}}
\def\earr{\end{array}}
\def\igdef{\ \stackrel{\mbox{\tiny{def}}}{=}\ }

\def\eps{\varepsilon}
\def\la{\lambda}
\def\al{\alpha}
\def\N{\mathbb{N}}
\def\R{\mathbb{R}}

\def\C{\mathbb{C}}
\def\I{\mathbb{I}}

\def\T{\mathbb{T}}

\def\cH{\mathcal{H}}

\def\cT{{\cal T}}
\def\cM{{\cal M}}
\def\cB{{\cal B}}
\def\cN{{\cal N}}
\def\cV{{\cal V}}
\def\cU{{\cal U}}

\def\ca{\mathbf{a}}

\def\inc{\subseteq}

\def\rai{^{1/2}}

\def\da{^\downarrow}

\DeclareMathOperator{\Preal}{Re} 
 \DeclareMathOperator{\tr}{tr}
\DeclareMathOperator{\gen}{span}

\DeclareMathOperator{\leqp}{\leqslant}

\newcommand{\mat}{\mathcal{M}_d(\mathbb{C})}

\newcommand{\matsad}{\mathcal{H}(d)}

\newcommand{\matud}{\mathcal{U}(d)}

\newcommand{\matpos}{\mat^+}

\def\beq{\begin{equation}}
\def\eeq{\end{equation}}

\def\pausa{\medskip\noi}

\def\Ax2{\,( S_{E(\cF)^\#_\cV})\hat{}_x }

\newcommand{\tcal}{\T_{d}(\ca)}

\newcommand{\asubi}{\ca=(a_i)_{i\in\I_k}\,\in(\R^k_{>0})\da}

\newcommand{\phisn}{\Phi_{(N,\, S,\, \mu)}}

\newcommand{\then}{\Rightarrow}

\newcommand{\paren}[1]{\left(#1\right)}
\newcommand{\llav}[1]{\left\{#1\right\}}
\newcommand{\abs}[1]{\left|#1\right|}
\newcommand{\norm}[1]{\left\|#1\right\|}

\usepackage{color}

\definecolor{rojo}{rgb}{1,0,0}
\definecolor{azul}{rgb}{0,0,1}

\begin{document}

\title{Local Lidskii's theorems for unitarily invariant norms}
\author{ Pedro G. Massey $^{*}$, Noelia B. Rios $^{*}$ and Demetrio Stojanoff 
\footnote{Partially supported by CONICET 
(PIP 0150/14), FONCyT (PICT 1506/15) and  FCE-UNLP (11X681), Argentina. } \ 
 \footnote{ e-mail addresses: massey@mate.unlp.edu.ar, nbrios@mate.unlp.edu.ar, demetrio@mate.unlp.edu.ar}
\\ {\small Depto. de Matem\'atica, FCE-UNLP
and IAM-CONICET, Argentina  }}

\date{}
\maketitle
\begin{abstract}
Lidskii's additive inequalities (both for eigenvalues and singular values) can be interpreted
as an explicit description of global minimizers of functions that are built on unitarily invariant norms, with
domains consisting of certain orbits of matrices (under the action of the unitary group). In this paper, we show
that Lidskii's inequalities actually describe all global minimizers of such functions and that local minimizers 
are also global minimizers. We use these results to obtain partial results related 
to local minimizers of generalized frame operator distances in the context of finite frame theory.
\end{abstract}



\tableofcontents

\section{Introduction}

Lidskii's additive inequalities \cite{Lidskii}  are ubiquitous in matrix analysis. They are part of 
the fundamental toolkit to deal with some of the most natural problems in this theory, such as matrix approximation problems 
(matrix nearness problems) and singular values/eigenvalues inequalities (see \cite{Bhat,HJ1,HJ2} and the references therein). 
Lidskii's inequalities are expressed in terms of an important pre-order between 
real vectors called majorization. Since majorization is intimately related to tracial inequalities 
involving convex functions, Lidskii's inequalities can be used to describe the structure of 
matrices that are optimal with respect to families of entropic-like functionals (see \cite{dnp,mr2010,mrs2,mrs3}). 
Lidskii's inequalities also provide some simple relations between the spectra of the sum of selfadjoint matrices and 
its summands, related to the solution of Horn's conjecture \cite{A.Horn} on the spectra of the sum of selfadjoint matrices, based on the work
of A. Klyachko \cite{Klyach} and A. Knutson and T. Tao \cite{KT} (see \cite{Fulton} for a historical account and a comprehensive description of 
the solution of Horn's conjecture).

\pausa
In the present paper, we consider local versions of Lidskii's inequalities with respect to unitarily invariant norms (u.i.n.).
To be more precise, consider a strictly convex u.i.n., denoted by N, on $\mat$ - the algebra of complex $d\times d$ matrices -
and fix a selfadjoint matrix $S\in \mat$. Fix $\mu \in \R^d$ and let 
$\mathcal O_\mu=\{	U^*\, D_\mu\, U:\ U\in\matud\}$, where $\matud$ denotes the group of unitary matrices and $D_\mu\in\mat$ denotes the diagonal
matrix with main diagonal $\mu$. Then, we consider $$\Phi:\mathcal O_\mu\rightarrow \R_{\geq 0} \peso{given by} \Phi(G)=N(S-G)\,.$$
Using Lidskii's additive inequality for eigenvalues of selfadjoint matrices, we can construct
$G^{\rm op}\in \mathcal O_\mu$ such that $\Phi(G^{\rm op})\leq \Phi(G)$, for every $G\in\mathcal O_\mu$
(see Section \ref{sec prelis nuis} for details). That is, Lidskii's inequality allows us to construct (explicitly) global minimizers
of $\Phi$. It is natural to wonder about the structure of all possible minimizers of $\Phi$ in $\mathcal O_\mu$. Moreover, since 
$\mathcal O_\mu$ has a natural metric (induced by the spectral norm) then we can ask about the structure of local minimizers of 
$\Phi$ in $\mathcal O_\mu$. These local minimizers arise naturally when considering optimization of 
$\Phi_2(G)=\|S-G\|_2$, i.e. when $N$ is the Frobenius norm. In this case, $\Phi_2^2$ is a smooth function defined on
a smooth manifold and thus we can apply (adapted) gradient descent algorithms to find minimizers of $\Phi_2$; notice that local minimizers of $\Phi_2$
are stability points of these algorithms and therefore their structure becomes part of the convergence analysis of these methods.
Thus, our first main problem is to study the structure of global and local minimizers of $\Phi$, for a general strictly convex u.i.n. $N$.
We carry out a similar analysis for Lidskii's inequality for singular values. In both cases we show that local minimizers are indeed global minimizers and
we compute their geometrical properties.

\pausa
Finite frame theory is a well established and rapidly growing area of research (see \cite{FinFram}). 
It is well known by now that several fundamental results of finite frame theory are counter-parts of 
well known results in matrix analysis. For example, the so-called frame design problem with prescribed frame operator and norms 
- that has played a central role in finite frame theory - is equivalent to some formulations of 
the Schur-Horn theorem (see the survey \cite{BowJas} and the reference therein). 
So it is no surprise that our results have implications in this area of research. Indeed, 
from the local version of Lidskii's theorem we 
derive some partial results related to the structure of local minimizers of the generalized frame operator 
distance (G-FOD) (see \cite{AC,dnp,strawn}).

\pausa
The paper is organized as follows. In Section \ref{sec prelis nuis} we recall several results from matrix analysis
that we use throughout the paper. In section \ref{lidskii nuis} we state and prove our main results
related to the local versions of Lidskii's theorem. Indeed, in Section \ref{sec 3.1} 
we obtain complete results showing that local minimizers of functions that are built
on strictly convex u.i.n's (as above) are global minimizers. 
In Section \ref{sec 3.2} we consider the
corresponding problem for Lidskii's singular value inequalities. In order to obtain these results, we consider
some (differential) geometrical properties of some auxiliary smooth maps.
In Section \ref{sec 4} we apply the results from the previous sections to the study of local
minimizers of the G-FOD induced by strictly convex u.i.n's.
We obtain some partial results regarding the general structure of these local minimizers and show that under
some further hypothesis, they are global minimizers of the G-FOD.

\section{Preliminaries}\label{sec prelis nuis}

In this section we introduce the notations, terminology and results from matrix analysis 
that we will use throughout the 
paper (see the texts \cite{Bhat,HJ1,HJ2}).

\pausa
{\bf Notation and terminology}. We let $\mathcal M_{k,d}(\C)$ be the space of complex $k\times d$ matrices and write $\mathcal M_{d,d}(\C)=\mat$ for the algebra of complex $d\times d$  matrices. We denote by $\matsad\subset \mat$ the real subspace of selfadjoint matrices and by $\matpos\subset \matsad$ the cone of positive semidefinite matrices. We let $\matud\subset \mat$ denote the group of unitary matrices.
For $d\in\N$, let $\I_d=\{1,\ldots,d\}$. 
Given a vector $x\in\C^d$ we denote by $D_x$ the diagonal matrix in $\mat$ whose main diagonal is $x$.
Given $x=(x_i)_{i\in\I_d}\in\R^d$ we denote by $x\da=(x_i\da)_{i\in\I_d}$ the  
vector obtained by rearranging the entries of $x$ in non-increasing order. We denote by 
$(\R^d)\da=\{x\da:\ x\in\R^d\}$ and $(\R_{\geq 0}^d)\da=\{x\da:\ x\in\R_{\geq 0}^d\}$. Given a matrix $A\in\matsad$ we denote by 
$\la(A)=\la(A)\da=(\la_i(A))_{i\in\I_d}\in (\R^d)\da$ the eigenvalues of $A$ counting multiplicities and arranged in 
non-increasing order.   
For $B\in\mat$ we let $s(B)=\la(|B|)$ denote the singular values of $B$, i.e. the eigenvalues of $|B|=(B^*B)^{1/2}\in\matpos$; we also let $\sigma(B)\subset \C$ denote the spectrum of $B$.
If $x,\,y\in\C^d$ we denote by $x\otimes y\in\mat$ the rank-one matrix given by $(x\otimes y) \, z= \langle z\coma y\rangle \ x$, for $z\in\C^d$.

\pausa Next we recall the notion of majorization between vectors, that will play a central role throughout our work.
\begin{fed}\rm 
Let $x\in\R^k$ and $y\in\R^d$. We say that $x$ is
{\it submajorized} by $y$, and write $x\prec_w y$,  if
$$
\suml_{i=1}^j x^\downarrow _i\leq \suml_{i=1}^j y^\downarrow _i \peso{for every} 1\leq j\leq \min\{k\coma d\}\,.
$$
If $x\prec_w y$ and $\tr x = \sum_{i=1}^kx_i=\sum_{i=1}^d y_i = \tr y$,  then $x$ is
{\it majorized} by $y$, and write $x\prec y$.
\end{fed}

\begin{rem}\label{desimayo}
\pausa Given $x,\,y\in\R^d$ we write
$x \leqp y$ if $x_i \le y_i$ for every $i\in \mathbb I_d \,$.  It is a standard  exercise 
to show that: 
\begin{enumerate}
\item $x\leqp y \implies x^\downarrow\leqp y^\downarrow  \implies x\prec_w y $. 
\item $x\prec y\implies |x|\prec_w|y|$, where $|x|=(|x_i|)_{i\in\I_d}\in\R_{\geq 0}^d$.
\item $x\prec y,\, |x|\da=|y|\da \implies x\da=y\da$. 
\item If $\tr(x)=\sum_{i\in\I_d} x_i=t$ then $\frac{t}{d}\,\uno_d\prec x$.
\EOE
\end{enumerate}
\end{rem}

\pausa 
Although majorization is not a total order in $\R^d$, there are several fundamental inequalities in 
matrix theory that can be described in terms of this relation. As an example of this phenomenon we can consider 
Lidskii's (additive) inequality for eigenvalues of sums of hermitians (see \cite{Bhat,HJ1,HJ2}). In the following result we also include the characterization of the case of equality obtained in \cite{mrs2}.

\begin{teo}[Lidskii's inequality]\label{mrs284}\rm
Let $A,\,B\in \matsad$ with eigenvalues $\la(A)$, $\la(B)\in (\R^d)^\downarrow$ respectively. Then 
\ben
\item $\la(A)-\la(B) \prec \la(A-B)$.
\item $\paren{\la(A)-\la(B)}^{\downarrow}=\la(A-B)$ if and only if there exists 
$\{v_i\}_{i\in\I_d}$ an ONB of $\C^d$ such that 
\beq
A=\sum_{i\in\I_d}\la_i (A) \ v_i\otimes v_i \py B=\sum_{i\in\I_d}\la_i (B)\ v_i\otimes v_i\  .
\eeq
Notice that in this case, $A$ and $B$ commute.
\qed\een 
\end{teo}
\pausa
Recall that a norm $\nui{\cdot}$ in $\mat$ is unitarily invariant if 
$$ \nui{UAV}=\nui{A} \peso{for every} A\in\mat \py U,\,V\in\matud\,.$$
Examples of unitarily invariant norms (u.i.n.) are the spectral norm $\|\cdot\|$ and the $p$-norms $\|\cdot\|_p$, for $p\geq 1$.
It is well known that majorization relations between singular values of matrices are intimately related 
with inequalities with respect to u.i.n's. 
The following result summarizes these relations (see for example \cite{Bhat}):

\begin{teo}\label{teo intro prelims mayo}\rm
Let $A,\,B\in\mat$ be such that $s(A)\prec_w s(B)$. Then:
\ben 
\item For every u.i.n. $N$ in $\mat$
we have that $N(A)\leq N(B)$.
\item If we assume that there exists a strictly convex
u.i.n. $N$ in $\mat$ such that $N(A)=N(B)$ then we have that $s(A)=s(B)$.
\een
\qed
\end{teo}

\section{Local Lidskii's theorems for unitarily invariant norms}\label{lidskii nuis}

Lidskii's additive inequalities (both for eigenvalues and singular values) can be interpreted
as an explicit description of global minimizers of functions that are built on unitarily invariant norms and whose
domains consist of certain orbits of matrices (under the action of the unitary group). In this section, we show
that Lidskii's inequalities actually describe all global minimizers of such functions, and that local minimizers 
are also global minimizers. This last fact will play a central role in the next section, in which we state and study Strawn's 
generalized conjecture.

\subsection{Selfadjoint matrices - eigenvalues}\label{sec 3.1}
\pausa
We begin with the following comments related to
 the classical Lidskii's inequality. 
Fix $S\in\matsad$ and $\mu\in(\R^d)\da$, and consider $\cO_\mu$  given by 
\beq\label{defi omu App nuis}
\cO_\mu=\{G\in\matsad:\la(G)=\mu\}=\{U^* D_{\mu}\,U:\ U\in\matud\} 
\eeq
We consider the usual metric in $\cO_\mu$ induced by the operator norm; hence $\cO_\mu$ is a metric space.

\pausa
For $N$ a strictly convex u.i.n., let
\beq\label{defi nui App}
 \Phi=\phisn:\cO_\mu\rightarrow \R_{\geq 0} \peso{be given by} \Phi(G)=N(S-G). 
\eeq
Using an ONB of eigenvectors of $S$ we can construct $G^{\rm op}\in \cO_\mu$ such that 
$\lambda(S-G{\rm op})=(\la(S)-\mu)\da$. By Lidskii's inequality and Remark \ref{desimayo}, we see that for every 
$G\in\cO_\mu$ we have that 
\beq\label{eq submayo rel}
 \la(S-G^{\rm op})\prec \la(S-G) \implies  s(S-G^{\rm op})=|\la(S-G^{\rm op})|\prec_w |\la(S-G)|=s(S-G)\,. 
\eeq
Hence, Theorem \ref{teo intro prelims mayo} implies that $\Phi(G^{\rm op})=N(S-G^{\rm op})\leq N(S-G)=\Phi(G)$, for $G\in\cO_\mu$.
Therefore, $G^{\rm op}$ is a global minimizer of $\Phi$ in $\cO_\mu$. Conversely, let $G\in \cO_\mu$ be a global minimizer of 
$\Phi$ in $\cO_\mu$. The previous comments together with item 3. in Remark \ref{desimayo} show that 
\beq\label{eq submayo rel2}
\la(S-G^{\rm op})\prec \la(S-G) \py  N(S-G^{\rm op})= N(S-G) \implies \la(S-G^{\rm op})= \la(S-G)\eeq
where we have used the fact that $N$ is strictly convex, the submajorization relation in Eq. \eqref{eq submayo rel} 
and Theorem \ref{teo intro prelims mayo}. In turn, Eq. \eqref{eq submayo rel2} together with Theorem \ref{mrs284} imply that there exists
 an ONB $\{v_i\}_{i\in\I_d}$
of $\C^d$ such 
$$
S=\sum_{i\in\I_d}\la_i \ v_i\otimes v_i \py G=\sum_{i\in\I_d}\mu_i\ v_i\otimes v_i\,,
$$where $(\la_i)_{i\in\I_d}=\la(S)\in (\R^d)\da$; that is, the global minimizer $G$ is obtained from $S$ as $G^{\rm op}$.

\pausa
It is then natural to ask about the structure of local minimizers $G_0$ of the map $\Phi$ in $\cO_\mu$, which is our main problem in this section. 
As we will see, these 
local minimizers are actually global minimizers of $\Phi$ (see Theorem \ref{teo LLTApp} below).

\begin{fed}\label{muchas defis mil nuis}
Let $S,G_0\in\matsad$. We consider 
\begin{enumerate}
\item The product manifold $\matud\times \matud$ endowed with the metric $$d((U_1,V_1),(U_2,V_2))=\max\{\|I-U_1^*U_2\|,\,\|I-V_1^*V_2\| \}\,.$$
\item 
$\Gamma=\Gamma_{(S,G_0)}:\matud\times \matud\rightarrow \matsad_\tau\stackrel{\rm def}{=}\{M\in\matsad:\ \tr(M)=\tau\}$ for $\tau=\tr(S)-\tr(G_0)$, given by
$$ \Gamma(U,V)=U^* S\,U-V^*G_0\,V\peso{for} U,\, V\in\matud\,.$$
\item For a given u.i.n. $N$ on $\mat$, we consider 
$\Delta_{(S,G_0)}^{N}=\Delta:\matud\times \matud\rightarrow \R_{\geq 0}$: 
$$\Delta(U,V)=N(\Gamma(U,V)) \peso{for} U,\, V\in\matud\,.$$
\EOE\end{enumerate}
\end{fed}
\pausa
Our motivation for considering the previous notions comes from the following:
\begin{lem}
\label{lem equiv de los probs}
Let $S\in\matsad$, $\mu\in(\R^d)\da$, $G_0\in\cO_\mu$ and consider the notations from Definition \ref{muchas defis mil nuis}. Given a u.i.n. $N$ on $\mat$, the following conditions are equivalent:
\begin{enumerate} 
\item $G_0$ is a local minimizer of $\Phi$ in $\cO_\mu$ (defined in Eq. \eqref{defi nui App});
\item $(I,I)$ is a local minimizer of $\Delta$ on $\matud\times \matud$.
\end{enumerate}
\end{lem}
\begin{proof}

1.$\implies$2. Consider $(U,W)\in \matud\times \matud$ such that 
$$d((U,W),(I,I))=\max\{\,\|I-U^*\|\,,\,\|I-W^*\| \, \}:=\varepsilon\,.$$
Hence, if $Z=WU^*\in\matud\,$ then $ U^* S\,U-W^* G_0\,W=U^*( S- Z^* G_0\, Z)\,U$. 
Notice that 
$$\|Z-I\|=\|W \,(U^*-W^*)\|\leq \|U^*-I\|+\|I-W^*\|\leq 2\,\varepsilon\ \implies$$ 
$$\Delta(U,W)
=N(U^*( S- Z^* G_0\, Z)\,U )= \Phi(Z^*G_0Z)
\peso{with} \| Z^*G_0Z-G_0\|\leq 4\,\varepsilon\|G_0\|\,.$$

2.$\implies$1. This is a consequence of the fact that the map $\matud\ni Z\mapsto 
Z^* G_0\,Z\in\cO_\mu$ is open 
(see, for example,  \cite[Thm. 4.1]{AS}  or \cite{DF}).
\end{proof}

\pausa
In what follows, given $\mathcal S\subset\matsad$ we consider the commutant of $\mathcal S$, denoted $S '$, 
that is the unital $^*$-subalgebra of $\mat$ given by 
$$ \mathcal S'=\{\ C\in\mat:\ [C,D]=0\ \text{ for every }\ D\in\mathcal S\ \}\subset \mat\,,$$
where $[C,D]=CD-DC$ denotes the commutator of $C$ and $D$.

\pausa
Recall that $\matud$ has a natural smooth (differential) manifold structure. Hence, we can consider $\matud\times \matud$ as a smooth manifold, endowed with 
the product structure. 
\begin{lem}\label{lema gama sobre}
Consider the notations from Definition \ref{muchas defis mil nuis}. Then 
$$
\Gamma
 \peso{is a submersion at} (I,I) \quad \iff \quad 
\{S,\,G_0\}'=\C\cdot I \ .
$$
\end{lem}
\begin{proof}
The (exponential) map $\matsad\ni X\mapsto \exp(X)$ 
allows us to identify the tangent space $\cT_{I}\matud$ with 
$i\cdot\matsad$. Since we consider the product structure on $\matud\times \matud$ we conclude that the differential of $\Gamma$ satisfies
$$D_{(I,I)}\Gamma (X,0)=[S,X] \py D_{(I,I)}\Gamma (0,X)=[X,G_0] \peso{for} X\in i\cdot \matsad\,. $$ 
Therefore $\Gamma$ is  not a submersion at $(I,I)$ 
if and only if there exists $0\neq Y\in\cT\matsad_\tau=\matsad_0$ 
(i.e. $Y\in\matsad$ such that $\tr\, Y=0$)
such that 
\beq\label{orto} 
\tr(Y\,[S,Z])=\tr(Y\,[Z,G_0])=0 \peso{for every} 
Z\in i\cdot \matsad \ . 
\eeq
Since $\tr(Y\,[S,Z])=\tr([Y,S]\,Z)$ and similarly 
$\tr(Y\,[Z,G_0])=\tr(Z\,[G_0,Y])$,  we see that in this case 
$$[Y,S]=0 =[G_0,Y]\in i\cdot \matsad\ .$$
Moreover, since $Y\neq 0$ and 
$\tr \, Y = 0$, then $Y$ has some non-trivial spectral projection $P$  which also satisfies that $[P,S]=[P,G_0]=0$. Conversely, in case there exists a non-trivial projection $P$ such that 
$[P,S]=[P,G_0]=0$, we can construct 
$ Y = \frac{P}{\tr\, P}  - \frac{I-P}{\tr \,(I- P)} $ 
so that $\tr\, Y =0$. Then 
$0\neq Y\in\cT\matsad_\tau$ and it satisfies Eq. \eqref{orto}, so that this matrix $Y$  is orthogonal to the range of the operator $D_{(I,I)}\Gamma$.
\end{proof}

\begin{pro}\label{pro al final commutan nomas}
Consider the notations from Definition \ref{muchas defis mil nuis} and assume that $N$ is a strictly convex u.i.n. If $(I,I)$ is a local minimizer of $\Delta$ in $\matud\times \matud$ then $[S,G_0]=0$.
\end{pro}
\begin{proof}
Assume that $[S,G_0]\neq 0$. Then there exists a minimal projection $P$ of the unital $^*$-subalgebra $\mathcal C=\{S,\,G_0\}'\subseteq\mat$ such that 
$[P\,S,P\,G_0]\neq 0$.  Indeed, $I\in\mathcal C$ is a projection such that $[I\,S,I\,G_0]\neq 0$. If $I$ is not a minimal projection in $\mathcal C$ then
there exists $P_1,\,P_2\in\mathcal C$ non-zero projections such that $I=P_1+P_2$; hence $[P_iS,P_iG_0]\neq 0$ for $i=1$ or $i=2$. If the corresponding $P_i$ is not minimal in $\mathcal C$ we can repeat the previous argument (halving) applied to $P_i$. Since we deal with finite dimensional algebras, the previous procedure finds a minimal projection $P\in\mathcal C$ as above. By applying a convenient change of orthonormal basis we can assume that 
$R(P)=\text{span}\{e_i: i\in\I_r\}$, where $r=\rk(P)>1$. Since $P$ reduces both $S$ and $G_0$
 we can consider $S_1=S|_{R(P)}\in\mathcal H(r)$ and 
$G_1=G_0|_{R(P)}\in\mathcal H(r)$. 
By minimality of $P$ we conclude that $\{S_1,\,G_1\}'=\C I_r \subset \cM_r(\C) $.
Using the case of equality of Lidskii's inequality (see Theorem \ref{mrs284}), we conclude that 
$$ b:=(\la(S_1)-\la(G_1))\da\prec a:=\la(S_1-G_1) \py a\neq b\,.$$
If we let $\sigma=\tr(S_1-G_1)$ then, by Lemma \ref{lema gama sobre} the map 
$$\cU(r)\times \cU(r)\ni(U,V)\mapsto U^*S_1\,U-V^*G_1\,V\in\cH(r)_\sigma $$
is a submersion at $(I_r,I_r)$. In particular, for every open neighborhood $\mathcal N $ of $(I_r,I_r)$ in $\cU(r)\times \cU(r)$ the set 
$$ \cM:=\{ U^*S_1\,U-V^*G_1\,V: \ (U,V)\in\mathcal N \}$$ contains an open neighborhood of $S_1-G_1$ in $\cH(r)_\sigma$. 
Consider $\rho:[0,1]\rightarrow (\R_{\geq 0}^{r})\da$ given by $\rho(t)=(1-t)\,a+t\,b$ for $t\in[0,1]$. 
Notice that $\rho(t)\prec a$ and $\rho(t)\neq a$ for $t\in (0,1]$. 
If we let $S_1-G_1=W^*D_{a}\,W$ for $W\in\cU(r)$ then the continuous curve $T(\cdot):[0,1]\rightarrow \cH(r)_\sigma$ given by 
$T(t)=W^*D_{\rho(t)}\,W$ for $t\in[0,1]$ satisfies that $T(0)=S_1-G_1$, $\la(T(t))\prec a$ and $\la(T(t))\neq a$ for $t\in(0,1]$.
Therefore, there exists $t_0\in(0,1]$ such that $T(t)\in\cM$ for $t\in[0,t_0]$ so, in particular, there exists $(U,V)\in\mathcal N$ such that 
$$ T(t_0)=U^*S_1\,U-V^*G_1\,V  \implies
\Delta(U\oplus P^\perp,V\oplus P^\perp)<\Delta(I_d,I_d)  \ , 
$$
because $N$ is a strictly convex u.i.n., 
where $U\oplus P^\perp,\, V\oplus P^\perp\in\matud$ act as the identity on $R(P)^\perp\subset \C^d$.
Since $\cN$ was an arbitrary neighborhood of $(I_r,I_r)$ we conclude that $(I_d,I_d)$ is not a local minimizer of $\Delta$ in $\matud\times \matud$, 
which contradicts Lemma \ref{lem equiv de los probs}.
\end{proof}

\begin{teo}[Local Lidskii's theorem]\label{teo LLTApp}
Let $S\in\matsad$ and $\mu=(\mu_i)_{i\in\I_d}\in(\R^d)\da$. Assume that 
$N$ is a strictly convex u.i.n. and that $G_0\in\cO_\mu$ is a local minimizer of 
$\Phi=\phisn$ on $\cO_\mu\,$. Then, there exists an ONB 
$\{v_i\}_{i\in\I_d}$ of $\C^d$ such that 
\beq\label{eq base buenapp1}
S=\sum_{i\in\I_d}\la_i\ v_{i}\otimes v_{i} \py  G_0=\sum_{i\in\I_d}\mu_i\ v_{i}\otimes v_{i}\,,
\eeq 
where $(\la_i)_{i\in\I_d}=\la(S)\in (\R^d)\da$. 
 In particular, $\la(S-G_0)=(\la(S)-\la(G_0))\da$ so $G_0$ is also a global minimizer of $\Phi$ on $\cO_\mu\,$.
\end{teo}
\begin{proof}
By Lemma \ref{lem equiv de los probs} and Proposition \ref{pro al final commutan nomas} we conclude that $[S,G_0]=0$. 
Notice that in this case there exists $\cB=\{v_i\}_{i\in\I_d}$ an ONB of $\C^d$ such that 
$$
S=\sum_{i\in\I_d} \la_i \ v_i\otimes v_i
\ , \ G_0=\sum_{i\in\I_d} \nu_i \ v_i\otimes v_i \ \text{ with }\ \la=(\la_i)_{i\in\I_d}\in (\R_{\geq 0}^d)\da \ ,
$$ 
for some $\nu_1,\ldots,\nu_d\in \R$.
We now show that under a suitable permutation of the elements of $\I_d$ we can obtain a representation as in Eq. \eqref{eq base buenapp1} above.
Indeed, assume that $j\in\I_{d-1}$ is such that $\nu_j<\nu_{j+1}$. If we assume that $\la_j>\la_{j+1}$ then consider the continuous curve of unitary operators $U(t):[0,\pi/2)\rightarrow \matud$ given by 
$$U(t)=\sum_{i\in\I_d\setminus\{j,\,j+1\}}v_i\otimes v_i + \cos(t)\ (v_j\otimes v_j 
+v_{j+1}\otimes v_{j+1}) +\sin(t)\ (v_j\otimes v_{j+1}- v_{j+1}\otimes v_j)\ , \ \ t\in[0,\pi/2)\,.$$
Notice that $U(0)=I_d$. We now define the continuous curve $G(t)=U(t) \,G_0\,U(t)^* \in\cO_\mu$, for $t\in[0,\pi/2)$.
Then $G(0)=G_0$ and we have that  
\beq\label{eq SFt}
 S-G(t)=\sum_{i\in\I_d\setminus\{j,\,j+1\}} (\la_i-\nu_i) \ v_i\otimes v_i + \sum_{r,s=1}^2 \gamma_{r,s}(t) \ v_{j+r}\otimes v_{j+s}\,,\eeq
where $M(t)=(\gamma_{r,s}(t))_{r,s=1}^2$ is determined by  
$$ M(t)
= \begin{pmatrix} \la_j & 0 \\0 & \la_{j+1}\end{pmatrix}
- V(t)
\begin{pmatrix}\nu_{j} & 0 \\0 & \nu_{j+1}\end{pmatrix}
V(t)^* \py V(t)=
\begin{pmatrix}\cos(t) & \sin(t) \\-\sin(t) & \cos(t)\end{pmatrix} 
\ , \ \ t\in[0,\pi/2)\,.
$$
Let us consider 
\beq\label{defi R}R(t)=
V^*(t) \begin{pmatrix}
\la_j - \la_{j+1} & 0 \\
0 & 0
\end{pmatrix}V(t)
- 
\begin{pmatrix}
\nu_j & 0 \\
0 & \nu_{j+1}
\end{pmatrix}
\implies M(t)=V(t) \, R(t)\, V^*(t) + \la_{j+1}\, I_2\,.
\eeq
We claim that $\la(R(t))\prec \la(R(0))$ and $\la(R(t))\neq \la (R(0))$ for $t\in(0,\pi/2)$ (i.e., the majorization relation is strict). Indeed, since 
$R(t)$ is a curve in $\cH(2)$ such that $\tr(R(t))$ is constant, it is enough to show that the function $[0,\pi/2)\ni t\mapsto \tr(R(t)^2)$ is strictly decreasing in $[0,\pi/2)$. 
Using that $\la_j-\la_{j+1}>0$ we have that
$$ V^*(t) \begin{pmatrix}\la_j - \la_{j+1} & 0 \\ 0 & 0 \end{pmatrix}V(t) = g(t)\otimes g(t) \peso{where} g(t)=(\la_j-\la_{j+1})^{1/2}(\cos(t),\sin(t)) \ , \ \ t\in[0,\pi/2)\,.
$$ If $D\in\cM_2(\C)$ is the diagonal matrix with main diagonal $(\nu_j,\,\nu_{j+1})$ then 
$R(t)=g(t)\otimes g(t)-D$ so 
$$\tr(R(t)^2)=\tr((g(t)\otimes g(t))^2)+\tr(D^2)-2\,\tr(g(t)\otimes g(t)\ D)=c-2\,\langle D\,g(t),\,g(t)\rangle$$
where $c=\|g(t)\|^4+\nu_j^2+\nu_{j+1}^2=(\la_j-\la_{j+1})^2+\nu_j^2+\nu_{j+1}^2\in\R$ is a constant and 
$$\langle D\,g(t),\,g(t)\rangle =(\la_j-\la_{j+1})\ (\cos^2(t) \, \nu_j+\sin^2(t)\,\nu_{j+1})$$ is strictly increasing in $[0,\pi/2)$, since $\nu_j<\nu_{j+1}$. Thus, $\la(R(t))\prec \la(R(0))$ and $\la(R(t))\neq \la (R(0))$ for $t\in(0,\pi/2)$. Hence, by Eq. \eqref{defi R}, we see that  
$$ \la(M(t))=\la(R(t))+\la_{j+1}\,\uno_2 \implies \la(M(t))\prec \la(M(0)) \ , \ \la(M(t))\neq \la(M(0)) \ ,\ \ t\in(0,\pi/2)\,.$$ 
Then, using Eq. \eqref{eq SFt} and Theorem \ref{teo intro prelims mayo}, for $t\in(0,\pi/2)$ 
$$
\la(S-G(t))\prec \la(S-G_0) \ \ , \ \ \la(S-G(t))\neq \la(S-G_0) \then N(S-G(t))<N(S-G_0).
$$
This last inequality, which is a consequence of the assumption $\la_j<\la_{j+1}$, contradicts 
the local minimality of $G_0$ in $\cO_\mu$. Hence, since $\la_j\leq \la_{j+1}$ we see that $\la_j=\la_{j+1}$; in this 
case, we can consider the basis $\cB'=\{v_i'\}_{i\in\I_d}$ obtained by transposing the vectors 
$v_j$ and $v_{j+1}$ in the basis $\cB$. In this case $S\ v_i'=\la_i\ v_i'$ for $i\in\I_d$, $G_0\ v_i=\nu_i\ v_i'$ for 
$i\in\I_d\setminus\{j,\,j+1\}$ and $G_0\ v_j'=\nu_{j+1}\, v_j'$, $G_0 \ v_{j+1}'=\nu_{j}\, v_{j+1}'$. 
After performing this argument at most $d$ times we get the desired ONB.
\end{proof}

\subsection{Arbitrary matrices - singular values}\label{sec 3.2}

In this section we obtain results related to a local Lidskii's theorem for arbitrary matrices with respect to singular values.
As a consequence, we characterize the case of equality in the classical Lidskii's inequality for singular values. 

\pausa
Recall that if $A,B\in\mat$ then Lidskii's singular value inequality states that 
\begin{equation}\label{lidskii vs sing}
\abs{s(A)-s(B)}\prec_w s(A-B).
\end{equation}
In what follows we fix $A\in\mat$, $s\in(\R^d_{\geq 0})\da$, $s\neq 0$, and $N$ a strictly convex u.i.n. We consider the set of matrices whose vector of singular values is $s$, i.e.
$$
\cV_s:=\llav{C\in\mat:\, s(C)=s},
$$ endowed with the usual metric, induced by the spectral norm. We further consider the function 
$$\Psi_{(N,\,A,\,s)}=\Psi:\cV_s\rightarrow \R_{\geq 0} \peso{given by}
\Psi(C)=N(A-C).
$$ 
With an argument similar to that in the beginning of Section \ref{sec 3.1}, now 
based on the singular value decomposition (SVD) and Lidskii's inequality in Eq. \eqref{lidskii vs sing}, we can explicitly construct global minimizers
 of $\Psi$ on $\cV_s$. As before, we are interested in the structure of local minimizers of $\Psi$ in 
$\cV_s$. 

\pausa
We will describe the structure of local minimizers of 
$\Psi$ in $\cV_s$ and show that local minimizers are
 actually global minimizers. In order to do this we consider the following well known matrix construction: 
for $C\in\mat$, let $\widehat{C}\in\cH(2d)$ be given by
$$
\widehat{C}=
\begin{pmatrix}
0 & C\\
C^* & 0
\end{pmatrix}.
$$
Let $U,V\in \cU(d)$ be such that $C=V^* D_{s(C)} U$, 
 and define $W\in \cU(2d)$ given by
$$
W=\frac{1}{\sqrt{2}}
\begin{pmatrix}
V & U\\
-V & U
\end{pmatrix}.
$$
Then $\widehat{C}= W^*\, (D_{s(C)} \oplus-D_{s(C)}) \,W$, which implies that 
\beq\label{eq orden}
\la(\widehat{C})_i= \begin{cases}  \quad \quad
 s_i(C) & \peso{if} \ 1\leq i\leq d \\
-s_{d-i+1}(C)
& \peso{if}\ d+1\leq i\leq 2d\ .
\end{cases}  
\eeq

\medskip

\begin{fed}
Let $A\coma B \in \mat$, let $N$ be a u.i.n. on $\mat$ and $s\in (\R^d_{\geq 0})\da$, $s\neq 0$. We consider:
\begin{enumerate}
\item The real space $\mathcal S=\{\widehat C:\ C\in\mat\}\subset H(2d)$.
\item The map $\Pi_{(A\coma B)}=\Pi:\matud^4\rightarrow \mathcal S$ given by
\begin{eqnarray}\label{eq defi Pi}\Pi(U_1\coma U_2\coma V_1\coma V_2)&=&(U_1\oplus V_1)^* \, \widehat A  \ (U_1\oplus V_1) -
(U_2\oplus V_2)^* \,\widehat B  \ (U_2\oplus V_2)\\ \nonumber &=& \widehat{U_1^*\,A\,V_1-U_2^*\, B\, V_2}\,.
\end{eqnarray}
\item The map $\Xi_{(A\coma B)}:\matud^4\rightarrow \R_{\geq 0}$ given by 
\beq \label{eq defi Xi}
\Xi(U_1\coma U_2\coma V_1\coma V_2)=N(U_1^*\,A\,V_1-U_2^*\, B\, V_2)\,.
\eeq \EOE
\end{enumerate}
\end{fed}

\begin{lem}\label{lem equiv de los probs para nuis}
Let $A\in\mat$, $s\in(\R_{\geq 0}^d)\da$, $s\neq 0$, and $B\in\cV_s$. 
Given a u.i.n. $N$ on $\mat$, the following conditions are equivalent:
\begin{enumerate} 
\item $B$ is a local minimizer of $\Psi_{(N\coma A\coma s)}$ in $\cV_s$;
\item $(I,I,I,I)$ is a local minimizer of $\Xi_{(A\coma B)}$ on $\matud^4$.
\end{enumerate}
\end{lem}
\begin{proof} An argument similar to that in the proof of Lemma \ref{lem equiv de los probs} shows the equivalence of the items above.
\end{proof}

\pausa Next we develop some geometric properties of $\Pi$. As before, we consider $\matud^4$ as a smooth manifold, endowed with the product structure. 
\begin{lem}\label{lema Pi sobre} Let $A,\, B\in\mat$ and let $\Pi$ be as Eq. \eqref{eq defi Pi}.
Then the following conditions are equivalent:
\begin{enumerate}
\item $\Pi_{(A\coma B)}$ is a submersion at $(I,I,I,I)$;
\item Whenever $Z\in\mat$ is such that $A^*Z\coma AZ^*\coma B^*Z\coma BZ^*\in\matsad$, then $Z=0$.
\end{enumerate}
\end{lem}
\begin{proof}
Notice that since $\Pi$ is a smooth function, item 1. holds if and only if the differential map
$$D=D\Pi_{(I\coma I\coma I\coma I)}:(i\cdot \matsad)^4 \rightarrow \mathcal S\subset \mathcal H(2d)\peso{is surjective .}$$
We now check that $D$ is {\it not} surjective if and only if there exists $Z\in\mat$, $Z\neq 0$, such that 
$A^*Z\coma AZ^*\coma B^*Z\coma BZ^*\in\matsad$. Indeed, it is straightforward to compute
$$D(X_1\coma X_2\coma Y_1\coma Y_2)= \widehat{-X_1\,A+A\,Y_1  + X_2\,B-B\,Y_2 }\peso{for} X_1\coma X_2\coma Y_1\coma Y_2\in i\cdot\matsad\,.$$
Hence, $D$ is not surjective if and only if there exists $Z\in\mat$, $Z\neq 0$, such that 
\begin{equation}\label{eq cond perp}
\widehat Z\ \ \perp  \ \ \widehat{-X_1\,A+A\,Y_1  +X_2\,B-B\,Y_2 }\peso{for} X_1\coma X_2\coma Y_1\coma Y_2\in i\cdot\matsad\,.
\end{equation}
In this case (setting $X_2=Y_2=0$) we have that 
\begin{equation}\label{eq cond perp2}
0=\tr(\widehat Z \ (\widehat{-X_1\,A+A\,Y_1}))=2\,\Preal[\tr(Z^*(-X_1\,A+A\,Y_1))]
\peso{for} X_1\coma Y_1\in i\cdot\matsad\,.
\end{equation} 
Using that $\Preal[\tr(C)]=\tr(\Preal [C])$ and the tracial property, we see that Eq. \eqref{eq cond perp2} is equivalent to 
\begin{equation}\label{eq cond perp3}
0=\tr(X_1\,(AZ^*-ZA^*)) + \tr(Y_1\,(A^*Z-Z^*A)) \peso{for} X_1\coma Y_1\in i\cdot\matsad\,.
\end{equation}
Since $(AZ^*-ZA^*)\coma (A^*Z-Z^*A)\in i\cdot \matsad$, Eq. \eqref{eq cond perp3} holds if and only if 
$$AZ^*-ZA^*=0\py A^*Z-Z^*A=0 \implies AZ^*\coma A^*Z\in\matsad\,.$$
Similarly, by setting $X_1=Y_1=0$ in Eq. \eqref{eq cond perp} and arguing as before, we conclude that 
$BZ^*\coma B^*Z\in\matsad$.

\pausa
Conversely, assume that there exists $Z\in\mat$, $Z\neq 0$, such that $A^*Z\coma AZ^*\coma B^*Z\coma BZ^*\in\matsad$. 
Then, arguing as before, it follows that $Z$ verifies the perpendicularity condition in Eq. \eqref{eq cond perp}; thus, $D$ is not surjective in this case.	
\end{proof}

\begin{pro}\label{pro asumiendo que pi es sumersion}
Fix $A\in \mat$, a strictly convex u.i.n. $N$ on $\mat$ and $s\in (\R^d_{\geq 0})\da$, $s\neq 0$. If $B\in\cV_s$ be a local minimizer of
$\Psi=\Psi_{(N\coma A\coma s)}$ then, $\Pi_{(A\coma B)}$ is not a submersion at $(I\coma I\coma I\coma I)$.
\end{pro}
\begin{proof}
Assume that $\Pi_{(A\coma B)}$ is a submersion at $(I\coma I\coma I\coma I)$. 
Assume further that any of the conditions $A^*B\coma  A\,B^*\in\matsad$ does not hold.
In this case it is straightforward to check that $\widehat A$ and $\widehat B$ do not commute.
In particular, by Theorem \ref{mrs284} $$b:=(\la(\widehat A)-\la(\widehat B))\da\prec 
a:=\la(\widehat A - \widehat B) \py a \neq b\,.$$
Now, by Eq. \eqref{eq orden} we see that if we let 
$$\tilde b:=(s(A)-s(B)\coma -[s(A)- s(B)]) \ , \ \tilde a:=(s(A-B)\coma -s(A-B))\implies b=(\tilde b) \da \ , \ a=(\tilde a)\da\,. $$
Hence, if we let $\rho:[0,1]\rightarrow \R^{2d}$ be given by $\rho(t)= (1-t)\, \tilde a + t\, \tilde b$ , for $t\in [0,1]$ then: 
\begin{enumerate}
\item $\rho(t)\prec a$ and $\rho(t)\da \neq a$, for every $t\in (0,1]$ ;
\item $\rho(0)=\tilde a$ ;
\item For every $t\in[0,1]$ there exists $c_t\in \R^d$ such that $\rho(t)=(c_t\coma -c_t)$.
\end{enumerate}
In order to see item 1. above, recall that $$ \rho(t)=(1-t)\, \tilde a + t\, \tilde b\prec (1-t)\, (\tilde a)\da
 + t\, (\tilde b)\da= (1-t)\,  a + t\, b \prec a$$ and, $((1-t)\,  a + t\, b)\da=(1-t)\,  a + t\, b\neq a$ (since $a\neq b$), 
for $t\in (0,1]$. Consider a SVD for $A-B= V^*\, D_{s(A-B)} \, U$, for some $U\coma V\in\matud$, and define 
$$
W=\frac{1}{\sqrt{2}}
\begin{pmatrix}
V & U\\
-V & U
\end{pmatrix}\in \mathcal U(2d)\,.
$$
Then $\widehat{A-B}= W^*\, (D_{s(A-B)} \oplus-D_{s(A-B)}) \,W=W^*\ D_{\tilde a}\ W$; Let us consider 
$T(t)= W^* \, D_{\rho(t)}\, W$ for $t\in [0,1]$.
Then, using item 3. above, we see that $T(t)\in \mathcal S$ for $t\in [0,1]$. By the hypothesis on $\Pi_{(A\coma B)}=\Pi$, 
for every open neighborhood of $I\in \mathcal N\subset \matud$, the set
$$ \mathcal M=\{ \Pi(U_1\coma U_2\coma V_1\coma V_2):\ U_i\coma V_i\in\mathcal N\, , \ i=1,2\}$$
contains an open neighborhood of $\widehat{A-B}$ in $\mathcal S$. Since $T:[0,1]\rightarrow \mathcal S$ is 
a continuous curve such that $T(0)=\widehat{A-B}$, then there exists $t_0\in (0,1)$ such that $T(t)\in \mathcal M$, for 
$t\in [0,t_0]$. In particular, there exist $U_i\coma V_i\in\mathcal N$, for $i=1,2$ such that 
$$T(t_0)=\widehat{U_1^*\,A\, V_1 - U_2^*\, B\, V_2} \ \ , \ \ \la(T(t_0))=\rho(t)\da \prec a \ \ , \  \ \la(T(t_0))\neq  a\,.$$
Hence, $s( U_1^*\,A\, V_1 - U_2^*\, B\, V_2)\prec_w s(A-B)$ and $s( U_1^*\,A\, V_1 - U_2^*\, B\, V_2)\neq  s(A-B)$. Using that 
$N$ is a strictly convex u.i.n. we conclude that $$ \Xi_{(A\coma B)}(U_1\coma U_2\coma V_1\coma V_2)
=N(U_1^*\,A\, V_1 - U_2^*\, B\, V_2)< N(A-B)=\Xi_{(A\coma B)}(I\coma I\coma I\coma I)\,.$$
Since $\mathcal N$ is an arbitrary neighborhood of $I$ in $\matud$ we see that $(I\coma I\coma I\coma I)$ is not a local
minimizer of $\Xi_{(A\coma B)}$, which contradicts Lemma \ref{lem equiv de los probs para nuis}. 

\pausa
The previous argument shows that $A^*B\coma AB^*\in\matsad$. If we set $Z=B\in\mat$, we see that 
$$Z\neq 0\py A^*Z\coma AZ^*\coma B^*Z\coma BZ^*\in\matsad\,.$$ 
Now, Lemma \ref{lema Pi sobre} implies that $\Pi$ is not a submersion at $(I\coma I\coma I\coma I)$, which contradicts 
our assumption on $\Pi$; this last fact proves the result.
\end{proof}

\begin{rem}\label{rem sobre la condicion EY}
Let $A\coma B\in\mat$ be such that $A^*B,\, AB^*\in\matsad$. Then, Eckart and Young \cite{EY} claimed that there exist
 matrices $U,\, V\in \matud$ such that
$$U^* A\,V=A \py U^*B\,V=D_\beta \peso{with} \beta\in\R^d\,.$$
Indeed, notice that the hypothesis also holds for $X^*AY$ and $X^*BY$, for any $X\coma Y\in\matud$. Thus, by considering a SVD of $A$ and the previous 
comment, we can assume that
$A=\oplus_{i=1}^k\alpha_i\,I_i$ with $I_i\in\mathcal M_{d_i}(\C)$ the identity matrix, $d_1+\ldots+d_k=d$ and $\alpha_1>\ldots>\alpha_k\geq 0$. 
Let  $\C^d=\oplus_{i=1}^k \C^{d_i}$, and consider the block representation of $B$ with respect to this decomposition, 
$B=(B_{ij})_{i,j=1}^k$.
Under the previous assumption on $A$, we have that $A\,B, AB^*\in\matsad$; then, 
$A\,B=(A\,B)^*=B^*\,A$ and $AB^*=(AB^*)^*=BA$.
These equations imply that
$$ B_{ji}^*\,\alpha_j=\alpha_i\, B_{ij} \py  \alpha_i\,B_{ji}^*=B_{ij}\,\alpha_j \peso{for}1\leq i,j\leq k\,.$$
In particular, if $i\neq j$ and $\alpha_i\neq 0$ we get
$$\alpha_i \, B_{ij} = \alpha_j\, B_{ji}^*=\frac{\alpha_j^2}{\alpha_i}  B_{ij}\implies B_{ij}=0\,.$$
In case that $i\neq j$ and $\alpha_i=0$ then 
$\alpha_j\,B_{ij}=0 \implies B_{ij}=0$ because $\alpha_j\neq \alpha_i=0$. 
And if $\alpha_i\neq 0$ then
$$\alpha_i\, B_{ii}^*=\alpha_i \,B_{ii}\implies B_{ii}=B_{ii}^*\,.$$
Thus $B=\oplus_{i=1}^k B_{ii}$. Let note that if $\alpha_k=0$, the block $B_{kk}\in\cM_{d_k}(\C)$ is arbitrary. Consider now the unitary matrices
$U_i\in\mathcal U(d_i)$ such that $U^*B_{ii}U=D_{\gamma_i}$, with $\gamma_i\in\R^{d_i}$ for $\alpha_i\neq 0$ (that includes $1\leq i\leq k-1$), 
and eventually (when $\alpha_k=0$), a SVD $U_k^*B_{kk}V=D_{\gamma_k}$ for  
$U_k,\, V_k\in\mathcal U(d_k)$, with $\gamma_k\in \R_{\geq 0}^{d_k}$. Then, taking 
$$ U=\oplus_{i=1}^k U_i \py V=\oplus_{i=1}^{k-1}U_i\oplus V_k,$$ 
and setting $\beta=(\gamma_1\coma \ldots\coma \gamma_k)\in\R^d$, we get
$$ U^*AV=A \py  U^*BV=\oplus_{i=1}^k D_{\gamma_i}=D_{\beta}\,.$$
\EOE
\end{rem}

\begin{pro}\label{estruct de min loc de Lidskki para val sing}
Fix $A\in \mat$, a strictly convex u.i.n. $N$ on $\mat$ and $s\in (\R^d_{\geq 0})\da$, $s\neq 0$. 
Let $B\in\cV_s$ be a local minimizer of
$\Psi=\Psi_{(N\coma A\coma s)}$. Then, 
$A^*B\coma  A\,B^*\in\matsad$. 
\end{pro}
\begin{proof} We argue by induction on the dimension $d\geq 1$. Indeed, in case $d=1$ then
the result follows from the fact that, given $a\in\C$, any local minimizer $b$ of the function $f(c)=|a-c|$ for $c\in\{z\in\C:\ |z|=s>0\}$
satisfies that $\bar a\cdot b\in\R$, and then also $a\cdot \bar b\in\R$.

\pausa
We assume that the result holds for all dimension $\tilde d$ such that $1\leq \tilde d\leq d-1$. Let $A\coma B\in\mat$ be such that $B$ is a local minimizer 
of $\Psi$ in $\cV_s$. Notice that by Proposition \ref{pro asumiendo que pi es sumersion}, 
$\Pi_{(A\coma B)}$ is not a submersion at $(I\coma I\coma I\coma I)$. 
By Lemma \ref{lema Pi sobre}, we conclude that there exists $Z\in\mat$, $Z\neq 0$, such that 
$A^*\,Z\coma A\,Z^*\coma B^*\,Z\coma B\,Z^*\in\matsad$. Consider a SVD, $D_{s(Z)}=U^*\, Z\, V$, for $U,\, V\in \matud$.
By replacing $A$ and $B$ by $U^*\, A\, V$ and $U^*\, B\, V$ we can further assume that $Z=D_{s(Z)}$, 
where $s(Z)=(s_i(Z))_{i\in\I_d}\in (\R^d_{\geq 0})\da$. We let:
\begin{enumerate}
\item $\sigma(Z)=\{\sigma_1>\ldots>\sigma_k\}$ be the distinct eigenvalues of $Z=D_{s(Z)}\in \matpos$.
\item $I_j=\{i\in\I_d:\ s_i(Z)=\sigma_j\}$ and $m_j=\# (I_j)$, for $j\in\I_k$.
\end{enumerate}
Notice that since $Z\neq 0$ then $\sigma_1>0$. Using that $A^*Z\coma AZ^*\coma B^*Z\coma BZ^*\in\matsad$ with
 $Z=\oplus_{j\in\I_k} \sigma_j\, I_j$ and Remark \ref{rem sobre la condicion EY}, we conclude that: 
\beq \label{eq rep a,b,a-b}
A=\oplus_{j\in\I_k} A_j \py B=\oplus_{j\in\I_k} B_j\ \ \implies \ \ A-B=\oplus_{j\in\I_k} A_j -B_j \, ,
\eeq
where $I_j\coma A_j\coma B_j \in\mathcal M_{m_j}(\C)$, for $j\in\I_k$; moreover, $A_j\coma B_j \in\mathcal H(m_j)$, whenever $\sigma_k\neq 0$, 
for $j\in\I_k$. Using the fact that $B$ is a local minimizer of $\Psi$ on $\cV_s$ we see that 
$$B_j \ \text{is a local minimizer of} \ \Psi_{(N_j\coma A_j\coma s(B_j))} \  , \ \text{for}  \ j\in\I_k\, ,$$
where $N_j$ is the strictly convex u.i.n. on $M_{m_j}(\C)$ given by $N_j(C)=N(C\oplus 0_{d-m_j})$.
In turn, this last fact shows that for each $j\in\I_k$ for which $\sigma_j\neq 0$ - which includes all $1\leq j\leq \max\{k - 1\coma 1\}$ -
$B_j$ is a local minimizer of $\Phi_{(N_j\coma A_j\coma \la(B_j))}$; Theorem
\ref{teo LLTApp} shows that $A_j$ and $B_j$ commute, so
$A_j^*B_j\coma A_j\, B_j^*\in \mathcal H(m_j)$, for $j\in\I_k$ such that $\sigma_j\neq 0$.
Therefore, we consider two possible cases: on the one hand, if $\sigma_k\neq 0$ then
the previous remarks show that $$A^*\,B=\oplus_{j\in\I_k} A_j^*\,B_j\in \mathcal H(d)\,,$$ and similarly, $A\,B^*\in\matsad$.

\pausa
On the other hand, if $\sigma_k=0$, notice that $m_k=\dim \ \ker Z<d$ (since $Z\neq 0$), and 
$B_k\in M_{m_k}(\C)$ is a local minimizer of  $\Psi_{(N_k\coma A_k\coma s(B_k))}$. 
In this case we can apply the inductive hypothesis and conclude that $A_k^*\,B_k\coma A_k\, B_k^*\in\mathcal H(m_k)$.
Since we have already showed that $A_j^*\,B_j\coma A_j\, B_j^*\in\mathcal H(m_j)$, for $1\leq j\leq k-1$, we now see that 
$A^*\,B\coma A\, B^*\in\mathcal H(d)$.
\end{proof}
\begin{teo}\label{desc en vs simultanea}
Let $A\in\mat$, $s\in(\R^d_{\geq 0})\da$ and $N$ be a strictly convex u.i.n. 
If $B$ is a local minimizer of $\Psi$ in $\cV_s$ then
$A$ and $B$ have a joint SVD i.e.,
there exist $U,V\in\cU(d)$ such that 
$$
A=U^*D_{s(A)}V \py B=U^*D_{s(B)}V.
$$ In particular, $s(A-B)=|s(A)-s(B)|\da$ and $B$ is a global minimizer of $\Psi$ in $\cV_s$.
\end{teo}
\begin{proof}
 Notice that if $B$ is a local minimizer of $\Psi$ in $\cV_s$ and $X^*AY=D_{s(A)}$ is a SVD
 of $A$ for some $X,\, Y\in\matud$, we can replace  $A$ by $D_{s(A)}$ and $B$ by $X^*BY$ to get
$$
N(A-B)=N(D_{s(A)}-X^*BY).
$$
Since $\cV_s\ni C\mapsto X^*CY\in \cV_s$ is a homeomorphism of $\cV_s$ then we can assume,
 without loss of generality, that $A=D_{s(A)}$. By Proposition \ref{estruct de min loc de Lidskki para val sing} we get that $A\,B,\, A\,B^*\in\matsad$; 
then by \cite{EY} (see Remark \ref{rem sobre la condicion EY}) there exist matrices $U,\, V\in \matud$ such that
$$ U^* A\,V=D_{s(A)}\,(=A) \py U^*B\,V=D_\beta\peso{with} \beta \in\R^d\,. $$
Suppose now that $\beta\notin \R_{\geq 0}^{d}$, so there exists $1\leq \ell \leq d$ such that $\beta_\ell<0$.
Notice that the function $f(t):[0,\pi]\rightarrow \R_{\geq 0}$ given by 
$$f(t)=|s_\ell(A)- e^{it}\beta_\ell| \peso{for} t\in[0,\pi]$$ is strictly decreasing. Let
$W(t)=(w_{jk})_{j,\,k\,\in\I_d}\in\matud$ be the diagonal matrix whose main diagonal is  given 
 by $w_{jj}=1$ for all $j\neq \ell$, and $w_{\ell\ell}=e^{it}$ for $t\in[0,\pi]$;
hence $W(0)=I$. Define
$$
B(t)=U\,W(t)D_\beta V^* \peso{for} t\in [0,\pi]\,. 
$$  
Then $B(t)$ is a continuous curve such that $B(0)=B$, $B(t)\in \cV_{s}$ for $t\in [0,\pi]$, and 
$$
\Psi(B(t))= N(U\,(D_{s(A)} -D_{\beta(t)})\, V^*)=N(D_{|\alpha-\beta(t)|}) \ ,
$$ where $\beta(t)=s(B(t))$ for $t\in [0,\pi]$. Hence, $\beta_j(t)=\beta_j$ for $j\neq \ell$ and $\beta_\ell(t)=e^{it}\,\beta_\ell$. Therefore,
$|\alpha_j-\beta_j(t)|=\alpha_j-\beta_j$ is constant for $j\neq \ell$ and $|\alpha_\ell-\beta_\ell(t)|=f(t)$ for $t\in[0,\pi]$. Since $f$ is strictly decreasing, we conclude that $$|\alpha-\beta(t)|\prec_w  |\alpha-\beta| \ \implies \ \Psi(B(t)) \ \text{is strictly decreasing for} \ t\in [0,\pi]\,.$$ This last fact contradicts the assumption of $B$.
Therefore $\beta \in\R_{\geq 0}^{d}$. 

\pausa
Suppose now that $\beta\neq s=s(B)$ i.e. $\beta\neq \beta\da$; since $A=D_{s(A)}$ with $s(A)=s(A)\da$ then, 
by Theorem \ref{teo LLTApp}, $D_\beta$ is not a local minimizer 
of $\Phi=\Phi_{(N,\,D_{s(A)},\,s)}$ on $\cO_{s}\,$. Then, there exists a continuous curve
$\delta(t):[0,1]\rightarrow \matud$ such that $\delta(0)=I$ and 
$h(t)=N(D_{s(A)}-\delta(t)^*\, D_\beta\, \delta(t)\,)$,
is strictly decreasing in $[0,1]$. Therefore, if we let $\tilde B(t)=U\,\delta(t)^*\, D_\beta\,\delta(t)\,V^*$ for 
$t\in [0,1]$ then $\tilde B(0)=B$, $\tilde B(t)\in \cV_s$ for $t\in [0,1]$, and the function $\Psi(\tilde B(t))=h(t)$ 
is strictly decreasing in $[0,1]$. These facts contradict our assumption that $B$ is a local minimizer of $\Psi$ in $\cV_s$. 
Hence, $U^*BV=D_{s}$ and then, $s(A-B)=\abs{s(A)-s}\da$ which implies that $B$ is a global minimizer of $\Psi$ in $\cV_s$.
\end{proof}

\begin{cor}[Equality in Lidskii's inequality for singular values]\label{coro caract eq in desi Lidskii val sin} Let $A,\, B\in\mat$. 
Then $|s(A)-s(B)|\da=s(A-B)$ if and only if $A$ and $B$ have a joint SVD.
\end{cor}
\begin{proof} In case $A\coma B\in \mat$ have a joint SVD, then it is straightforward to show that $|s(A)-s(B)|\da=s(A-B)$.
Conversely, assume that $|s(A)-s(B)|\da=s(A-B)$ and choose your favorite strictly convex u.i.n. $N$ on $\mat$. By the comments 
at the beginning of this section, we see that $B$ is a global minimizer of $\Psi_{(N\coma A\coma s(B))}=\Psi$ on $\cV_{s(B)}$. In particular, $B$ is a
local minimizer of $\Psi$ in  $\cV_{s(B)}$; the result now follows from Theorem \ref{desc en vs simultanea}.
\end{proof}

\section{Application: Generalized Strawn's conjecture}\label{sec 4}

In this section we consider some problems within the theory of finite frames (see the texts \cite{FinFram, Chr}. 
for general references on this topic). It is worth pointing out that our results can be also be described as 
the solution to certain matrix nearness problems, following the scheme of \cite{High} (see Remark \ref{rem equiv probs nearness}).

\pausa
In what follows we adopt the following: 

\pausa
{\bf Notation and terminology}: let $\cF=\{f_i\}_{i\in\I_k}$ be a finite sequence in $\C^d$. Then,
\ben
\item $T_\cF\in \cM_{d,k}(\C)$ denotes the synthesis operator of $\cF$ given by $T_\cF\cdot(\alpha_i)_{i\in\I_k}=\sum_{i\in\I_k}\alpha_i\, f_i$.
\item $T_\cF^*\in \cM_{k,d}(\C)$ denotes the analysis operator of $\cF$ and it is given by $T_\cF^*\cdot f=(\langle f,f_i\rangle)_{i\in\I_k}$.
\item  $S_\cF\in \matpos$ denotes the frame operator of $\cF$ and it is given by $S_\cF=T_\cF\,T_\cF^*$. Hence, $Sf=\sum_{i\in\I_k} \langle f,f_i\rangle f_i=\sum_{i\in\I_k} f_i\otimes f_i (f)$ for $f\in\C^d$. 
\item We say that $\cF$ is a frame for $\C^d$ if it spans $\C^d$; equivalently, $\cF$ is a frame for $\C^d$ if $S_\cF$ is a positive invertible operator acting on $\C^d$.
\EOE\een

\subsection{Generalized frame operator distances}
Let $S\in\matpos$ and $\asubi$. In this case we consider 
$$
\tcal := \llav{ \cG=\{g_i\}_{i\in\I_k} \in  (\C^d)^{k} \  : 
\norm{g_{i}}^2=a_i  \,\mbox{, for every }   i \in \IN{k}}\,.
$$ By definition, $\tcal$ is the (Cartesian) product of spheres in $\C^d$; 
hence, we consider the product metric of the Euclidean metrics in each of these spheres, namely $$d(\cG\coma \cG')=\max\{\|g_i-g_i'\|:\ i\in\I_k\} \peso{for} 
\cG=\{g_i\}_{i\in\I_k},\, \cG'=\{g_i'\}_{i\in\I_k}\in\tcal\,.$$  Notice that 
$\tcal$ is a compact metric space with the product metric.
Given a strictly convex u.i.n $N$ on $\mat$, we can consider 
the generalized frame operator distance (G-FOD) in $\tcal$ (see \cite{dnp}) given by
$$\Theta_{(N,\, S,\, \ca)}=\Theta: \tcal \rightarrow \R_{\geq0} \peso{given by} \Theta(\cG)=N(S-S_{\cG})\,$$
where $S_\cG=\sum_{i\in\I_k} g_i\otimes g_i$ denotes the frame operator of a family $\cG\in\tcal$. 
This notion is based on the frame operator distance (FOD) $\Theta_{(\|\cdot\|_2,\, S,\, \ca )}$ introduced by Strawn in \cite{strawn}, where
$\|A\|_2^2=\tr(A^*A)$ denotes the Frobenius norm, $A\in\mat$. 

\pausa
Based on his work and on numerical evidence, Strawn conjectured in \cite{strawn} that 
local minimizers of $\Theta_{(\|\cdot\|_2,\, S,\, \ca )}:\tcal\rightarrow \R_{\geq 0}$ are global minimizers. In 
\cite{dnp} we settled Strawn's conjecture in the affirmative; indeed, we obtained the following results
related to the more general G-FOD induced by a strictly convex u.i.n.:
\begin{teo}[See\cite{dnp}]\label{teo dnp} Let $S\in\matpos$ and $\asubi$. Then, there exists $\nu^{\rm op}=
\nu^{\rm op}(S\coma \ca)\in (\R_{\geq 0}^d)\da$ (that can be computed explicitly) such that: 
\begin{enumerate}
\item There exists $\cG^{\rm op}\in \tcal$ such that $\lambda(S_{\cG^{\rm op}})=\nu$. In this case, if $N$ is a u.i.n. in $\mat$ then
\begin{equation}\label{eq nuopop}
\Theta_{(N,\, S,\, \ca)}(\cG^{\rm op})\leq \Theta_{(N,\, S,\, \ca)}(\cG) \quad \text{for} \ \cG\in\tcal\,.
\end{equation}
\item If $N$ is a strictly convex u.i.n. and $\cG_0$ is a global minimizer of $\Theta_{(N,\,S,\, \ca)}$ on $\tcal$ then $\lambda(S_{\cG_0})=\nu^{\rm op}$.
\item If $\cG_0$ is a local minimizer of $\Theta_{(\|\cdot\|_2,\, S, \,  \ca)}$ on $\tcal$ then $\lambda(S_{\cG_0})=\nu^{\rm op}$; hence $\cG_0$ 
is a global minimizer of $\Theta_{(\|\cdot \|_2,\, S,\, \ca )}$. 
\qed
\end{enumerate} 
\end{teo}
\pausa
We point out that Theorem \ref{teo dnp} is obtained in terms of a translation of G-FOD problems into frame completion 
problems with prescribed norms. Roughly speaking, given $S\in\matpos$ and $\asubi$ as above, there exists a family $\cF_0=\{f_i\}_{i\in\I_d}\in\C^d$
such that for each $\cG\in\tcal$ the operator difference 
\beq\label{eq dnp trans}S-S_{\cG}=\|S\|\,I-(S_{\cF_0}+S_\cG)\,.\eeq
Notice that if we let $\cF=(\cF_0\coma \cG)\in (\C^d)^{d+k}$ be the finite sequence obtained by juxtaposition of $\cF_0$ and $\cG$ then
$S_{\cF_0}+S_\cG=S_{\cF}$. In \cite{dnp} any such $\cF$ is called a completion of $\cF_0$ by a family $\cG$, with norms prescribed by the sequence $\ca$.
Eq. \eqref{eq dnp trans} can be used to show items 1. and 2. in Theorem \ref{teo dnp} for a u.i.n. $N$. In order to get information about
local minimizers of $\Theta_{(N,\, S, \, \ca)}$ from Eq. \eqref{eq dnp trans} we should assume further that $N$ is the Frobenius norm. 
This obstruction to the general case of item 3. (for a strictly convex u.i.n. $N$) seems to be a limitation of the 
reduction methods from \cite{dnp}. Hence, we state the following:
\begin{conj}\label{conjetura gfod}Given $S\in\matpos$ and $\asubi$ and a strictly convex u.i.n. $N$ in $\mat$, let 
$\Theta_{(N,\, S, \, \ca)}:\tcal\rightarrow \R_{\geq 0}$ be given by $\Theta_{(N,\, S, \, \ca)}(\cG)=N(S-S_\cG)$.
If $G_0$ is a local minimizer $\Theta_{(N,\, S, \, \ca)}$ in $\tcal$ then:
\begin{enumerate}
\item $\la(S-S_{\cG_0}) \prec \la(S-S_{\cG})$, for every $\cG\in\tcal$;
\item In particular, $\cG_0$ is a global minimizer of $\Theta_{(\tilde N,\, S, \, \ca)}$, for every u.i.n. $\tilde N$ on $\mat$.
\EOE
\end{enumerate}
\end{conj}
\pausa
We point out that item 2. in Conjecture \ref{conjetura gfod} is a consequence of item 1. Nevertheless, item 2. is directly related
with the possible applications of the solution of Conjecture \ref{conjetura gfod} for the G-FOD problems. 

\pausa
In what follows, we will describe the first features of local minimizers of $\Theta_{(N,\, S, \, \ca)}:\tcal\rightarrow \R_{\geq 0}$, 
for an arbitrary strictly convex u.i.n. $N$ in $\mat$. 
We will also show that Conjecture \ref{conjetura gfod} holds under some further hypothesis on the spectral structure of local minimizers.

\pausa
We end this section with the following remark, in which we show the connection between G-FOD problems and matrix nearness problems.
\begin{rem}\label{rem equiv probs nearness}
Let $S\in\matpos$ and consider a strictly convex u.i.n. $N$ in $\mat$. Let $\mu_j\in(\R^d)\da$, for $j\in\I_k$, and consider the orbits
$$\mathcal O_{\mu_j}=\{G\in\matsad:\ \lambda(G)=\mu_k\}\ , \ j\in\I_k\,.$$
We can then consider the matrix nearness problem (as described in \cite{High}, see also \cite{LiPoon})
\begin{equation}\label{eq equiv prob aprox} 
\text{argmin}\ \{N(S-H)):\ H\in\mathcal O_{\mu_1}+ \ldots + \mathcal O_{\mu_k} \}\,.
\end{equation}
Let $\asubi$ and consider the particular case: $\mu_j=a_j\, e_1$, for $j\in\I_k$, where $\{e_i\}_{i=1}^d$ denotes the canonical basis of $\C^d$. 
Then $G\in \mathcal O_{\mu_j}$ if and only if $G=g\otimes g$ for some $g\in\C^d$ with $\|g\|^2=a_j$, $j\in\I_k$. Hence, the matrix nearness problem
in Eq. \eqref{eq equiv prob aprox} coincides with the problem of computing global minimizers on $\Theta_{(N\coma S\coma \ca)}$ in $\mathbb T_d(\ca)$. Similarly,
the study of local minimizers of the matrix nearness problem corresponds to the study of local minimizers of $\Theta_{(N\coma S\coma \ca)}$. 
It is worth pointing out that for the Frobenius norm, local minimizers of the matrix nearness problem arise naturally as stability points of
(effective) gradient descent algorithms, as those considered in \cite{LiPoon}. Hence, settling Conjecture \ref{conjetura gfod} in the affirmative would be a 
relevant result from an applied point of view. \EOE

\end{rem}

\subsection{Properties of local minimizers of the G-FOD on $\tcal$}
In this section we consider the following

\begin{nota}\label{notaciones strawn}
Fix $S\in\matpos$,  $\asubi$ and a strictly convex u.i.n. $N$ on $\mat$. We consider
\begin{enumerate}
\item $\Theta_{(N,\, S, \, \ca)}=\Theta: \tcal \rightarrow \R_{\geq0}$ given by $\Theta(\cG)=N(S-S_{\cG})$.
\item A local minimizer $\cG_0=\llav{g_i}_{i\in\I_k}\in\tcal$ of $\Theta_{(N,\, S, \, \ca)}$, with frame operator 
$S_0=S_{\cG_0}$.
\item For $\mu\in(\R^d)\da$, the unitary orbit  $\cO_\mu$  given by 
$$\cO_\mu=\{G\in\matsad:\la(G)=\mu\}=\{U^* D_{\mu}\,U:\ U\in\matud\}\,, $$
with the usual metric, induced by the operator norm; 
\item The function $\Phi=\phisn:\cO_\mu\rightarrow \R_{\geq 0}$ given by 
$ \Phi(G)=N(S-G).$
\end{enumerate}
\end{nota}

\begin{teo}\label{teo applic1}
Consider Notation \ref{notaciones strawn}. Then,
\begin{enumerate}
\item $S-S_0$ and $g_j \otimes g_j$ commute, for $j\in\I_k$. Hence, $g_j$ is an eigenvector of $S-S_0$, for $j\in\I_k$.

\item There exists $\{v_i\}_{i\in\I_d}$ an ONB of $\C^d$ such that 
$$
S=\sum_{i\in\I_d} \la_i(S) \ v_i\otimes v_i
\py  S_0=\sum_{i\in\I_d} \la_i(S_0) \ v_i\otimes v_i  
\,.
$$ 
In particular, we have that $\la(S-S_0)=\big[\la(S)-\la(S_0)\,\big]\da$.
\end{enumerate}
\end{teo}
\begin{proof}
For $j\in\I_k$ define  $$S_{[j]}=S-\sum_{i\neq j} g_i \otimes g_i\,\in\cH(d) \py \mu_{[j]}=a_j e_1\,\in\R^d_{\geq 0}\,.$$
Then, $\cO_{\mu_{[j]}}=\llav{g\otimes g: \norm{g}^2=a_j}$
 and it is straightforward to check that $g_j \otimes g_j$ is a local minimizer of $\Theta_{(N,\, S_{[j]},\, \mu_{[j]})}$ in $\cO_{\mu_{[j]}}$. Thus, 
by Theorem \ref{teo LLTApp},
 $g_j \otimes g_j$ commutes with $S_{[j]}$, for $j\in\I_k$. 
This last fact implies that $S-S_0$ and $g_j \otimes g_j$ commute, for $j\in\I_k$, which proves item 1.

\pausa
Since $\cG_0$ is a local minimizer of $\Theta$ in $\tcal$, there exists $\varepsilon>0$ such that 
\beq\label{eq loc para Phi}
U \in B_{(I\coma \eps)}\igdef  
\{U\in\cU(d): \|I-U\|<\epsilon\}\  \implies  \ \Phi_{(N,\, S,\, \mu)}(U\,S_{0}\,U^*)\geq  \Phi_{(N,\, S,\, \mu)}(S_{0})\, , 
\eeq where we are using Notation \ref{notaciones strawn}, with $\mu=\lambda(S_0)$. Indeed, let $\varepsilon>0$ be such that for 
$\cG'\in \tcal$ with $d(\cG_0,\cG')<\varepsilon$ we have that $\Theta(\cG')\geq \Theta(\cG_0)$. Notice that if $U\in 
B_{(I\coma \eps)}$ then $U\cdot\cG_0=\{U\,g_i\}_{i\in\I_k}\in\tcal$ is such that $d(\cG_0\coma U\cdot\cG_0)<\epsilon$. 
Therefore, $$ 
\Phi_{(N,\, S,\, \mu)}(U\,S_{0}\,U^*)=\Theta(U\cdot\cG_0)
\geq \Theta(\cG_0)=\Phi_{(N,\, S,\, \mu)}(S_{0})\,.
$$ 
Now, the map 	$\pi : \matud \to \cO_\mu$ given by $\pi(U) = U\,(S_{0})\,U^*$ is open 
(see \cite[Thm 4.1]{AS}), so that $\pi (B_{(I\coma \eps)})$ is an open 
neighborhood of $S_{0}$ in $\cO_\mu\,$, and 
$S_{0}$ is a local minimum for the map $\phisn$ on $\cO_{\mu}\,$. 
Item 2 now follows from
Theorem \ref{teo LLTApp} and the fact that $\mu=\la(S_0)\in(\R^d)\da$.
\end{proof}

\begin{cor}\label{cor consecuencias1}
Consider Notation \ref{notaciones strawn}. Let $W=R(S_0)\subset \C^d$; then,
\begin{enumerate}
\item $W$ reduces $S-S_0\in\cH(d)$; hence, $D:=(S-S_0)|_W\in L(W)$ is a selfadjoint operator; 
\item Let $\sigma(D)=\{c_1\coma \ldots\coma c_p\}$ be such that $c_1<c_2<\ldots<c_p$ and let 
$$J_j=\{\ell \in \I_k: \ D\,g_\ell=c_j\, g_\ell\} \peso{for} j\in\I_p\,.$$Then $\I_k$ is the disjoint union of $\{J_j\}_{j\in\I_p}$;
\item If we let $W_j=\text{span}\{g_\ell:\ \ell\in J_j\}$ then $W_j$ reduces both 
$S$ and $S_0$, for $j\in\I_p$. Moreover, $W=\oplus_{j\in\I_p} W_j$.
\end{enumerate}
\end{cor}
\begin{proof}
Notice that $W=\text{span}\{g_i:\ i\in\I_k\}$; on the other hand, by Theorem \ref{teo applic1}, $g_i$ is an eigenvector of $S-S_0$, for each $i\in\I_k$. 
These two facts show that $W$ is an invariant subspace of $S-S_0$; since $S-S_0$ is selfadjoint, $W$ reduces $S-S_0$. Thus, the restriction
$D=(S-S_0)|_W\in L(W)$ is a well defined selfadjoint operator acting on $W$. The previous remarks also show that 
$\I_k$ is the disjoint union of $\{J_i\}_{i\in\I_p}$.

\pausa
Let $j,\,\ell \in \I_p$ with $j\neq \ell$ and let $r\in J_j$ and $s\in J_\ell$. Then, $g_r\perp g_s$, since these vectors are eigenvectors 
of a selfadjoint operator, corresponding to different eigenvalues. Hence, $W_j\perp W_\ell$ and
$$ S_0 \,g_r=\sum_{u\in J_j} \langle g_r\coma g_u\rangle\ g_u\in W_j \,.$$ Thus, in particular, $W_j$ reduces $S_0$; using that $W_j$ also reduces $S-S_0$ we conclude that 
$W_j$ reduces $S=(S-S_0)+S_0$, for $j\in\I_p$. On the other hand, since $W=\sum_{j\in\I_p}W_j$ then
$W=\oplus_{j\in\I_p} W_j$.
\end{proof}

\begin{teo}\label{gen teo gen 4.6}
Consider Notation \ref{notaciones strawn}. Let $W=R(S_0)$ and let $\sigma((S-S_0)|_W)=\{c_1\coma\ldots\coma c_p\}$ as in Corollary 
\ref{cor consecuencias1}. Let $j\in \I_p$ and assume that there exists $c\in \sigma(S-S_0)$ such that $c_j<c$. Then, the family $\{g_j\}_{j\in J_j}$ is linearly independent.
\end{teo}
\begin{proof}
Suppose that for some $j\in\I_p$ the family $\{g_i\}_{i\in J_j}$ is linearly dependent.
Hence there exist coefficients $z_l\in \C$, $l\in J_j$ (not all 
zero) such that every $|z_l|\leq 1/2$ and  
\begin{equation}\label{eq1}
\sum_{l\in J_j}\overline{z_l} \ a_l\rai \ g_l=0\ . 
\end{equation}
Let $I_j \inc J_j$ be given by $I_j =\{l\in J_j:\ z_l\neq 0\}$. Assume that there exists 
$c\in\sigma(S-S_0)$ such that $c_j<c$
and let $h\in \C^d$ be such that $\|h\|=1$ and  $(S-S_0)\, h=c \,h$. 
For $t\in (-1/2,1/2)$ let 
$\cG(t)=\{g_i(t)\}_{i\in\IN{k}}$ be given by 
$$
g_l(t) = \begin{cases}  \ (1-t^2\,|z_l|^2)^{1/2} g_l+t\,z_l\,a_l\rai \, h 
& \mbox{if} \ \ l\in I_j\,;  \\
\quad \quad\quad g_l & \mbox{if} \ \ l\in \I_k\setminus I_j  \  .
\end{cases}  
$$ Notice that $\cG(t)\in \tcal $ for $t\in (-1/2,1/2)$.
 Let $\Preal(A)= \frac {A+A^*}{2}$ denote the real part of 
$A \in \mat$. For $l\in I_j$ then 
$$
g_l(t)\otimes g_l(t)
=(1-t^2\,|z_l|^2)\ g_l\otimes g_l+ t^2\,|z_l|^2\,a_l \ h\otimes h 
+ 2\,(1-t^2\,|z_l|^2)^{1/2}\,t \ \Preal(h\otimes \overline{z_l}\, a_l\rai\, g_l)
$$
Notice that $\cG(t)$ is a continuous curve in $\tcal$ such that $\cG(0)=\cG_0$.
Let $S(t)$ denote the frame operator of $\cG(t)\in \tcal $, so that $S(0)=S_0$, and let 
$T(t)=S-S(t)$ for $t\in (-1/2\coma 1/2)$. 
Note that 
$$
T(t)=S-S_0+ t^2 \sum_{l\in I_j }  |z_l|^2 \left(  g_l\otimes g_l - a_l \ h\otimes h \right) + R(t)
$$ 
where $R(t)=- 2 \suml_{l\in I_j }(1-t^2\,|z_l|^2)^{1/2}\,t \ \Preal(h\otimes a_l\rai\,\overline{z_l}\, g_l)$. 
Then $R(t)$ is a smooth function such that 
$$
R(0) = 0 \ \ , \ \ 
R'(0)=- \sum_{l\in I_j } \Preal(h\otimes \overline{z_l}\, a_l\rai\,g_l)
=- \Preal(h\otimes \sum_{l\in I_j } \overline{z_l}\,a_l\rai\, g_l) \stackrel{\eqref{eq1}}{=} 0 \ ,
$$ 
and such that $R''(0)=0$. Therefore 
$\lim\limits_{t\rightarrow 0} \ t^{-2}\ R(t)=0 $.
We now consider 
$$
V=\gen\,\big(\,\{g_l:\ l\in I_j \}\cup \{h\}\,\big)
=\gen\,\big\{\,g_l:\ l\in I_j \,\big\}\stackrel{\perp}{\oplus} \C\cdot h\ .
$$
Then $\dim V=s+1$,  for $s=\dim\gen\{g_l:\ l\in I_j \}\geq 1$. 
By construction, the subspace $V$ reduces $S-S_0$ and $T(t)$ in such a way that $(S-S_0)|_{V^\perp}=T(t)|_{V^\perp}$,
for $t\in(-1/2\coma 1/2)$. On the other hand 
\beq\label{adet}
T(t)|_{V}=(S-S_0)|_V+t^2 \sum_{l\in I_j }  |z_l|^2 \left(  g_l\otimes g_l 
- a_l \ h\otimes h \right) + R(t) = A(t)+R(t)\in L(V)\ ,
\eeq
where we use the fact that the ranges of the selfadjoint operators 
in the second and third term in the formula above clearly lie in $V$. 
Then $\la\big((\,S-S_0)|_V\,\big)=\big(\,c\coma   c_j \, \uno_s\, \big)
\in (\R^{s+1}_{>0})\da $ and 
$$\barr{rl}
\la \Big(\, \sum_{l\in I_j }  |z_l|^2  g_l\otimes g_l\,\Big) &
=(\gamma_1 \coma \ldots \coma \gamma_s \coma 0)
\in (\R^{s+1}_{\geq 0})\da \peso{with} \gamma_s>0 \ ,\earr
$$
where we have used the definition of $s$ and the fact that $|z_l|>0$ for $l\in I_j \,$ (and the known fact that 
if $	S\coma T\in \matpos \implies R(S+T) = R(S)+R(T)\,$). 
Hence, for sufficiently small $t$, 
the spectrum of the operator  $A(t)\in L(V)$
defined in Eq. \eqref{adet} is 
$$\barr{rl}
\la\big(\, A(t)\,\big) 
&=\big(\, c
-t^2 \, \sum_{l\in I_j }a_l\,|z_l|^2 \coma c_j+t^2\,\gamma_1 \coma \ldots \coma c_j+t^2 \,\gamma_s \,\big) \in (\R^{s+1}_{\geq 0})\da \ , \earr	
$$ 
where we have used the fact that $\langle g_l \coma h\rangle=0$ for every $l\in I_j \,$. 
Let us now consider 
$$
\la\big(\, R(t)\,\big)
=\big(\,\delta_1(t) \coma \ldots \coma \delta_{s+1}(t)\, \big) 
\in (\R^{s+1}_{\geq 0})\da\peso{for} t\in \R \ .
$$ 
Recall that in this case $\lim\limits_{t\rightarrow 0}t^{-2} \delta_j(t)=0$ 
for $1\leq j\leq s+1$. Using Weyl's inequality 
on Eq. \eqref{adet},  we now see that 
\beq\label{eq pro nueva 0}
\lambda \big(\,T(t)|_V\,\big)\prec \la\big(\, A(t)\,\big)
+\la\big(\, R(t)\,\big)\igdef \rho(t)\in (\R^{s+1}_{\geq 0})\da\,.\eeq We know that 
$$
\barr{rl}
\rho(t)&= \big(\, c-t^2 \, \sum_{l\in I_j }a_l\,|z_l|^2  +\delta_{1}(t)\coma c_j+t^2\,\gamma_1+\delta_2(t) \coma 
\ldots \coma c_j+t^2 \,\gamma_s+\delta_{s+1}(t) \, \big) 
\\ &\\
&=  
\Big(\, c-t^2 \, (\sum_{l\in I_j }a_l\,|z_l|^2  +\frac{\delta_{1}(t)}{t^2})\coma  
c_j+t^2\,(\gamma_1+\frac{\delta_2(t)}{t^2}) \coma \ldots \coma 
c_j+t^2 \,(\gamma_s+\frac{\delta_{s+1}(t)}{t^2}) \,\Big)  \ .
\earr
$$
Since by hypothesis $c_j<c$ then, the previous remarks show that there exists $\varepsilon>0$ such that 
if $t\in(0,\varepsilon)$ then, for every $i\in \I_s$ 
$$
c>c-t^2 \, (\sum_{l\in I_j }a_l\,|z_l|^2+\frac{\delta_{1}(t)}{t^2})> 
c_j+t^2(\gamma_{i}+\frac{\delta_{i+1}(t)}{t^2}) >c_j\,.$$
The previous facts show that for $t\in(0,\varepsilon)$ then
 $\rho(t) \prec 
\lambda((S-S_0)|_V)=\big(\,c\coma   c_j \, \uno_s\, \big)$ strictly. 
Therefore,
\begin{eqnarray*}
\lambda(T(t))&=&\big(\lambda((S-S_0)|_{V^\perp})\coma T(t)|_V\,\big)\da \stackrel{\eqref{eq pro nueva 0}}{\prec} 
\big(\lambda((S-S_0)|_{V^\perp})\coma \rho(t)\,\big)\\ &\prec& \big(\lambda((S-S_0)|_{V^\perp})\coma \lambda(S-S_0)|_V\,\big)\da=\lambda(S-S_0)\,,
\end{eqnarray*}
where the second majorization relation is strict (i.e. 
$(\lambda((S-S_0)|_{V^\perp})\coma \rho(t))\da\neq (\lambda((S-S_0)|_{V^\perp})\coma \lambda(S-S_0)|_V)\da$). 
Since $N$ is strictly convex,  
for every $t\in(0,\varepsilon)$ we have that  
$$
\Theta(\cG(t))=N(T(t))<N(S-S_0)=\Theta(\cG)\,.
$$
This last fact contradicts 
the assumption that $\cG_0$ is a local minimizer of $\Theta$ in $\tcal $.
\end{proof}

\subsection{Some special cases of Conjecture \ref{conjetura gfod}}

Consider Notation \ref{notaciones strawn} and assume that $k\geq d$; if we let $W=R(S_0)\subset \C^d$ then, 
as shown in Corollary \ref{cor consecuencias1}, $W$ reduces the self-adjoint 
operator $S-S_0\in\matsad$. In this section we show that in case $W$ is an eigenspace of $S-S_0$ then Conjecture \ref{conjetura gfod} holds for $\cG_0$ i.e., 
$\cG_0$ is a global minimizer of $\Theta_{(N\coma S\coma \ca)}$
in $\tcal$. In order to tackle this particular case, we introduce the following

\begin{rem}[A naive model]\label{un modelo naivo}
Fix $S\in\matpos$,  $\asubi$ and a strictly convex u.i.n. $N$. We let $t=\tr(\ca)=\sum_{i\in\I_k}a_i>0$.
 If $\cG\in\tcal$ then it is clear that 
$$ S_\cG\in \matpos \py \tr(S_\cG)=\sum_{i\in \I_k} \|g_i\|^2=\tr(\ca)=t\,.$$
Hence, we consider 
\beq\label{defi N}
\matpos_t=\{ A:\ A\in\matpos\coma \tr(A)=t\} \supset \{S_\cG:\ \cG \in\tcal\}\,,
\eeq endowed with the metric induced by the operator norm.
Moreover, we consider the map 
\begin{equation}\label{eq defi D}
\cD_{(N,\,S,\, t)}=\cD:\matpos_t\rightarrow \R_{\geq 0} \peso{given by} \cD(A)=N(S-A)\,.
\end{equation}
By Eq. \eqref{defi N} we see that 
\beq\label{eq just N}
\min\{\cD(A):\ A\in\matpos_t\} \leq 
\min\{\Theta(\cG):\ \cG\in\tcal\}\,.
\eeq The inequality in Eq. \eqref{eq just N} can be strict. Yet, we will show that under some additional hypothesis equality holds
in Eq. \eqref{eq just N}. Moreover, since $\matpos_t$ is a (larger but) simpler set, we are able to compute those $A\in\matpos_t$ 
that attain the minimum in the left hand side of Eq. \eqref{eq just N} (see Theorem \ref{teo caso co-fea} below); these facts together will allow us to prove Conjecture \ref{conjetura gfod} in some special cases.
\qed
\end{rem}

\begin{teo}\label{teo caso co-fea}
Let $S\in \matpos$, $\la(S)=(\la_i)_{i\in \I_d}\in(\R_{\geq 0}^d)\da$, $t>0$ and let $N$ be a u.i.n. 
Consider $\llav{v_i}_{i\in\I_d}$ an ONB of $\C^d$ such that $S\, v_i=\la_i\ v_i$, for $i\in\I_d$.
Let $c\leq \la_1$ be uniquely determined by 
$\sum_{i\in\I_d} (\la_i-c)^+=t$ and set 
$$ A^{\rm op}= \sum_{i\in\I_d} (\la_i-c)^+ \,v_i\otimes v_i \in\matpos_t\peso{so that} \lambda(S-A^{\rm op})=(\min\{c\coma \lambda_i\})_{i\in\I_d}\in (\R^d)\da\,.$$
Then, $A^{\rm op}$ is a global minimizer of $\cD$, defined as in Eq. \eqref{eq defi D}.
\end{teo}
\begin{proof}
By construction we see that $\la(S-A^{\rm op})=(\min\{c\coma \la_i\})_{i\in\I_d}$. Let $A\in\matpos_t$ be arbitrary; we consider the following cases:

\pausa
In case $c\leq \la_d$ then we see that $\la(S-A^{\rm op})=c\,\uno_d$. Since $\tr(A)=t$, then
$\tr(\la(S-A))=\tr(S-A)=\tr(S)-t=\tr(\la(S-A^{\rm op}))$. Thus, in this case we have (see item 4. in Remark \ref{desimayo}) that $\la(S-A^{\rm op})=c\, \uno_d\prec \la(S-A)$.
Hence, we conclude that $\cD(A^{\rm op})=N(S-A^{\rm op})\leq N(S-A)=\cD(A)$.

\pausa
In case $c>\la_d$, there exists $r\in\I_{d-1}$ such that $\la_{r}\geq c>\la_r+1$. Then,
$$(\gamma_i)_{i\in\I_d}:=\la(S-A^{\rm op})=(c\,\uno_{r}\coma \la_{r+1} \coma \ldots\coma \la_d )\in(\R^d)\da\,.$$
 If we let $\la(A)=(\alpha_i)_{i\in\I_d}\in (\R^d_{\geq 0})\da$ then, 
by Lidskii's additive inequality, we get that 
\beq\label{ap defi delta}
(\delta_i)_{i\in\I_d}:= ((\la_i-\alpha_i)_{i\in\I_d})\da=(\la(S)-\la(A))\da\prec \la(S-A)\,.
\eeq
We now show that $(\gamma_i)_{i\in\I_d}\prec (\delta_i)_{i\in\I_d}$; 
by construction $\tr((\gamma_i)_{i\in\I_d})=\tr((\delta_i)_{i\in\I_d})$ that is
\begin{equation}\label{trazagam}
\tr(\gamma)=\sum_{j=1}^d \gamma_j= r\, c + \sum_{j=r+1}^d \la_j =
\tr(\delta)=\sum_{j=1}^d \delta_j= \sum_{j=1}^d (\la_j -\al_j)\,.
\end{equation}
 Thus, in order to show that 
$(\gamma_i)_{i\in\I_d}\prec (\delta_i)_{i\in\I_d}$ we need to prove that 
$\sum_{j=k}^d \gamma_j\geq \sum_{j=k}^d \delta_j$, for every $k\in\I_d$, since the vectors are arranged in non-increasing order.
Notice that $\la_i\geq \la_i-\alpha_i$, for every $i\in\I_d$; then, by Remark \ref{desimayo}, we conclude that 
$\la_i\geq \delta_i$, for $i\in\I_d$. This guarantees that, for $r+1\leq k\leq d$,
\begin{equation}\label{ultimos}
\sum_{j=k}^d \gamma_j= \sum_{j=k}^d \la_j\geq  \sum_{j=1}^k \delta_j.
\end{equation}
We now define $\beta=\sum_{j=r+1}^d (\gamma_j - \delta_j)$ and notice that Eq. \eqref{ultimos} shows that $\beta\geq 0$.
By Eq. \eqref{trazagam}, 
$$
\sum_{j=1}^r \delta_j = r\,(c+\beta/r),
$$ which implies that $(c+\beta/r)\uno_{r} \prec (\delta_{1},\cdots,\delta_r)$. Hence, if $1\leq k\leq r$ then
\beq\label{eq desi para lem ap1}
\sum_{j=k}^r \delta_j \leq (r-k+1)\,(c+\beta/r)\leq (r-k+1)\, c + \beta\,.
\eeq
Therefore, for $1\leq k\leq r$, 
\beq\label{eq desi para lem ap2}
\sum_{j=k}^{d} \gamma_j -\sum_{j=k}^{d} \delta_j 
= (r-k+1)\,c + \beta -\sum_{j=k}^r \delta_j 
\stackrel{\eqref{eq desi para lem ap1}}{\geq} 0\,.
\eeq
Then, Eqs. \eqref{trazagam}, \eqref{eq desi para lem ap1} and \eqref{eq desi para lem ap2} show that $\gamma\prec \delta$.
Finally, if $N$ is a (strictly convex) u.i.n. then 
$$N(S-A^{\rm op})=N(D_\gamma)\leq N(D_\delta)\stackrel{\eqref{ap defi delta}}{\leq} N(S-A)$$ so $A^{\rm op}$ is a global minimizer of $\cD$ in 
$\matpos_t$.
\end{proof}

\pausa
The next result verifies Conjecture \ref{conjetura gfod} under some additional assumptions on the spectral structure of
local minimizers. 

\begin{teo}\label{teo conj caso esp}
Let $S\in\matpos$,  $\asubi$, with $k\geq d$, and let $N$ be a strictly convex u.i.n. in $\mat$. Let $\cG_0=\{g_i\}_{i\in\I_d}$ be a local minimizer of 
$\Theta$ in $\tcal$ such that there exists $c_1\in \R$ that satisfies $(S-S_{\cG_0}) g_i=c_1\,g_i$, for $i\in\I_k$. 
Then there exists an ONB $\{v_i\}_{i\in\I_d}$ of $\C^d$ such that 
\beq\label{eq teo rep s y s0}
S=\sum_{i\in\I_d} \la_i\ v_i\otimes v_i \py S_{\cG_0}=\sum_{i\in\I_d} (\la_i-c_1)^+\ v_i\otimes v_i\,, 
\eeq
where $(\la_i)_{i\in\I_d}=\la(S)\in (\R_{\geq 0}^d)\da$. Moreover, $\la(S-S_{\cG_0})\prec \la(S-S_\cG)$ for $\cG\in\tcal$. In particular, $\cG_0$ is a global minimizer of $\Theta$ in $\tcal$.
\end{teo}
\begin{proof} Let $S_0=S_{\cG_0}$. By Theorem \ref{teo applic1} 
there exists an ONB $\{v_i\}_{i\in\I_d}$ of $\C^d$ such that 
\beq\label{eq rep s y s0}
S=\sum_{i\in\I_d} \la_i \ v_i\otimes v_i
\py  S_0=\sum_{i\in\I_d} \la_i(S_0) \ v_i\otimes v_i  
\,.
\eeq In particular, $\la(S-S_0)=(\la(S)-\la(S_0))\da$. Let $W=R(S_0)=\text{span}\{g_i:\ i\in\I_k\}$,
 which reduces $S-S_0$ by Corollary \ref{cor consecuencias1}.  Then, by hypothesis we have that $\sigma((S-S_0)|_W)=\{c_1\}$. We consider the following two cases:

\pausa
Assume that $W=\C^d$. In this case $\sigma(S-S_0)=\{c_1\}$ and therefore $\la(S-S_0)=c_1\,\uno_d$. Thus, $\la_i-\la_i(S_0)=c_1$ which implies that 
$\la_i(S_0)=(\la_i-c_1)^+$, for $i\in\I_d$.  Notice that for every $\cG\in\tcal$ we have that $\tr(S-S_\cG)=\tr(S)-\tr(\ca)$; then we see that 
$c_1\,\uno_d=\tr(\la(S-S_0))=\tr(S-S_\cG)$ which shows (see item 4. in Remark \ref{desimayo}) that $\la(S-S_0)\prec \la(S-S_{\cG})$ for every $\cG\in\tcal$. This last fact implies that 
$\Theta(S_{\cG_0})=N(S-S_0)\leq N(S-S_{\cG})=\Theta(\cG)$, for every $\cG\in\tcal$. Thus, $\cG_0$ is a global minimizer of $\Theta$ in $\tcal$.

\pausa
Assume now that $W\neq \C^d$. Hence, $d>\dim W=\text{span}\{g_i:\ i\in\I_k\}$ which shows that $\cG_0$ is a linearly dependent family, since $k\geq d$.
Then, Theorem \ref{gen teo gen 4.6} implies that $c\leq c_1$ for every $c\in\sigma(S-S_0)$. Let $1\leq r\leq d-1$ be such that $\dim W=r$. 
Hence, $\la(S_0)=(\la_{1}(S_0)\coma \ldots\coma \la_{r}(S_0)\coma 0_{d-r})$ and $W=\text{span}\{v_i: 1\leq i\leq r\}$. 
Therefore, using Eq. \eqref{eq rep s y s0} and the previous facts we conclude that 
$$S-S_0=\sum_{i=1}^r(\la_i - \la_i(S_0))   
\ v_i\otimes v_i + \sum_{i=r+1}^d \la_i\ v_i\otimes v_i  
= c_1\ \sum_{i=1}^r  \ v_i\otimes v_i +  \sum_{i=r+1}^d  \la_i\ v_i\otimes v_i $$
Thus, $\sigma(S-S_0)\ni \la_i\leq c_1$ and $\la_i(S_0)=0$, for $r+1\leq i\leq d$; hence, $\la_i(S_0)=(\la_i(S)-c_1)^+$, for $r+1\leq i\leq d$.
Then, $\la_i(S_0)=(\la_i-c_1)^+$ for $i\in\I_d$ and therefore we obtain the representation of $S_0$ as in Eq. \eqref{eq teo rep s y s0}.

\pausa
Notice that $c_1=\la_1-(\la_1-c_1)^+$, since $W\neq \{0\}$. This shows that $c_1\leq \la_1$. Moreover, if we let $\tr(\ca)=t>0$ then 
$$ \sum_{i\in\I_d}(\la_i-c_1)^+=\tr(S_{\cG_0})=\tr(\ca)=t\,.$$
Using Remark \ref{un modelo naivo} and Theorem \ref{teo caso co-fea} we now see that for every $\cG\in\tcal$ we have that 
$\la(S-S_{\cG_0})\prec \la(S-S_\cG)$; in particular,
$$\Theta(\cG)=N(S-S_\cG)=\cD(S_\cG) \geq \cD(S_0)=N(S-S_0) \peso{since} S_\cG\in\matpos_t\,.$$
Thus, $\cG_0$ is a global minimizer of $\Theta$ in $\tcal$.
\end{proof}

\pausa
{\bf Acknowledgment}. We would like to thank Professor Eduardo Chiumiento for fruitful conversations related to the content of this work.



\begin{thebibliography}{}

{\small 




\bibitem{AS} E. Andruchow and D. Stojanoff, Geometry of Unitary Orbits, J. Operator Theory {\bf 26} (1991), 25-41.


\bibitem{AC} J. Antezana, E. Chiumiento, Approximation by partial isometries and symmetric approximation of finite frames, 
J. Fourier Anal Appl (2017). https://doi.org/10.1007/s00041-017-9547-5.

\bibitem{Bhat} R. Bhatia,  Matrix Analysis,
 Graduate Texts in Mathematics, 169. Springer-Verlag, New York, 1997.


\bibitem{BowJas}
M. Bownik, J. Jasper, Existence of frames with prescribed norms and frame operator. Excursions in harmonic analysis. Vol. 4, 103-117, 
Appl. Numer. Harmon. Anal., Birkhäuser/Springer, Cham, 2015.




\bibitem{FinFram} P. G. Casazza and G. Kutyniok eds., Finite Frames: Theory and Applications. Birkhauser, 2012. xii + 483 pp.


\bibitem{Chr} O. Christensen, An introduction to frames and Riesz bases. Applied and Numerical Harmonic Analysis. Birkhäuser Boston, 2003. xxii+440 pp.

\bibitem{DF}  D. Deckard and L. A. Fialkow, Characterization of Hilbert space operators with unitary cross sections, 
J. Operator Theory {\bf 2} (1979), 153-158.

\bibitem{EY}  G: Eckart and G: Young, The approximation of one matrix by another of lower rank, Psychometrika {\bf 1} (1936), 211-218.

\bibitem{Fulton} W. Fulton, Eigenvalues, invariant factors, highest weights, and Schubert calculus, 
Bull. Amer. Math. Soc. (N.S.) 37 (2000), no. 3, 209-249. 

\bibitem{High} N.J. Higham, Matrix nearness problems and applications. Applications of matrix theory 
(Bradford, 1988), 1-27, Inst. Math. Appl. Conf. Ser. New Ser., 22, Oxford Univ. Press, New York, 1989. 

\bibitem{A.Horn} A. Horn, Eigenvalues of sums of Hermitian matrices, Pacific J. Math. 12 (1962), 225-241.

\bibitem{HJ1} R.A. Horn, C.R. Johnson, Matrix analysis. Second edition. Cambridge University Press, Cambridge, 2013. 

\bibitem{HJ2} R.A. Horn, C.R. Johnson, Topics in matrix analysis. Corrected reprint of the 1991 original. Cambridge University Press, Cambridge, 1994.

\bibitem{Klyach} A. A. Klyachko, Stable bundles, representation theory and Hermitian operators, Selecta
Math. 4 (1998), 419-445.

\bibitem{KT} A. Knutson and T. Tao, The honeycomb model of GLn(C) tensor products I: proof of the
saturation conjecture, J. Amer. Math. Soc. 12 (1999), 1055-1090, 

\bibitem{LiPoon} C.K. Li, Y.T. Poon, T. Schulte-Herbrüggen, 
Least-squares approximation by elements from matrix orbits achieved by gradient flows on compact Lie groups. Math. Comp. 80 (2011), no. 275, 1601-1621. 

\bibitem{Lidskii} V.B. Lidskii, On the characteristic numbers of the sum and product of symmetric matrices. (Russian) Doklady Akad. Nauk SSSR (N.S.) 75, (1950). 769-772.

\bibitem {dnp}
P. Massey, N. B. Rios, D. Stojanoff, Frame completions with prescribed norms: local minimizers and applications. 
Adv. Comput. Math. 44 (2018), no. 1, 51–86. 


\bibitem{MR0}P. Massey, M.A. Ruiz, Tight frame completions with prescribed norms. Sampl. Theory Signal Image Process. 7 (2008), no. 1, 1-13.


\bibitem{mr2010}
P. Massey, M. Ruiz; Minimization of convex functionals over frame operators. Adv. Comput. Math. 32 (2010), no. 2, 131-153.

\bibitem {mrs1}
P. Massey, M. Ruiz, D. Stojanoff; Optimal dual frames and frame completions for majorization. Appl. Comput. Harmon. Anal. 34 (2013), 201-223.
 
\bibitem {mrs2}
P.G. Massey, M.A. Ruiz, D. Stojanoff; Optimal frame completions.
Advances in Computational Mathematics 40 (2014), 1011-1042.

\bibitem {mrs3}
P. Massey, M. Ruiz, D. Stojanoff; Optimal frame completions with prescribed norms for majorization.  
J. Fourier Anal. Appl. 20 (2014), no. 5, 1111-1140.

\bibitem {strawn}
N. Strawn; Optimization over finite frame varieties and structured dictionary design, Appl. Comput. Harmon. Anal. 32 (2012) 413-434.

}
\end{thebibliography}
\end{document}